\documentclass[a4paper,11pt]{article}
\usepackage[latin1]{inputenc}
\usepackage{amsmath, amsthm}
\usepackage{amsfonts}
\usepackage{amssymb}
\usepackage{dsfont}
\usepackage{geometry}
\usepackage{color}
\usepackage{bm}
\usepackage{srcltx}
\usepackage{accents}

\newcommand{\e}{\varepsilon}

\newcommand{\di}{\,\mathrm{div}}
\newcommand{\DIV}{\,\mathrm{div}}
\newcommand{\R}{\mathbb{R}}
\newcommand{\N}{\mathbb{N}}
\newcommand{\C}{\mathcal}
\newcommand{\ol}{\overline}

\newcommand{\dx}{\,\mathrm dx}

\newcommand{\ds}{\,\mathrm ds}

\newcommand{\dxt}{\,\mathrm dx\,\mathrm dt}

\newcommand{\weaklim}{\rightharpoonup}

\newcommand{\CC}{\mathbf{C}}

\newcommand{\db}{\mathbf{d}}
\newcommand{\ub}{\mathbf{u}}
\newcommand{\vb}{\mathbf{v}}
\newcommand{\zb}{\mathbf{z}}

\newcommand{\fb}{\mathbf{f}}

\newcommand{\ukkt}{\mathbf{u}^{k-1,t}}

\newcommand{\Tb}{ \mathbf{T} }

\newcommand{\uu}{\mathbf{u}}

\newcommand{\bell}{\bm{\ell}}

\newcommand{\ph}{\varphi}
\newcommand{\divv}{\text{div}}

\newcommand{\io}{\int_\Omega}
\newcommand{\aein}{\text{ a.e. in }}

\newcommand{\capa}{\text{cap}}

\newcommand{\red}{\color{black}}
\newcommand{\black}{\color{black}}

\newcommand{\af}{{\mathfrak a}}

\newcommand{\Kc}{K}

\newcommand{\tA}{\tilde{\mathcal A}}

\newcommand{\ac}{\accentset{\circ}}
\newcommand{\triplenorm}{\ensuremath{| \! | \! |}}

\newtheorem{thm}{Theorem}[section]
\theoremstyle{plain}

\newtheorem{remark}[thm]{{\textbf Remark}}
\newtheorem{lemma}[thm]{{\textbf Lemma}}
\newtheorem{theorem}[thm]{{\textbf Theorem}}
\newtheorem{proposition}[thm]{{\textbf Proposition}}
\newtheorem{definition}[thm]{{\textbf Definition}}
\newtheorem{corollary}[thm]{{\textbf Corollary}}

\allowdisplaybreaks

\date{January 15, 2016}

\author{Christian Heinemann\footnote{Weierstrass Institute for Applied Analysis and Stochastics (WIAS), Mohrenstr. 39, 10117 Berlin (Germany),
	E-mail: \texttt{christian.heinemann@wias-berlin.de}}~
	and Kevin Sturm\footnote{Universit\"at Duisburg-Essen, Fakult\"at f\"ur Mathematik,
  Thea-Leymann-Stra\ss e 9, 45127 Essen (Germany),
	E-mail: \texttt{kevin.sturm@uni-due.de}}
	}

\begin{document}
\title{Shape optimisation for a class of semilinear variational inequalities 
	with applications to damage models}
\maketitle

\begin{abstract}
	The present contribution investigates
	shape optimisation problems
	for a class of semilinear elliptic variational inequalities
	with Neumann boundary conditions.
	Sensitivity estimates and material derivatives
	are firstly derived in an abstract operator setting where
	the operators are defined on polyhedral subsets of reflexive Banach
	spaces.
	The results are then refined for
	variational inequalities arising from minimisation problems
	for certain convex energy functionals
	considered over upper obstacle sets in $H^1$.
	One particularity is that we allow for dynamic obstacle functions which
	may arise from another optimisation problems.
	We prove a strong convergence property for the material derivative
	and establish state-shape derivatives under regularity assumptions.
	Finally, as a concrete application from continuum mechanics, we show how the dynamic obstacle case can be used
	to treat shape optimisation problems for time-discretised brittle damage models
	for elastic solids.
	We derive a necessary optimality system for optimal shapes
	whose state variables approximate desired damage patterns and/or displacement fields.
\end{abstract}

  \vspace{4mm}

\noindent {\it AMS subject clas\-si\-fi\-ca\-tion:}
49J27,   		
49J40,   		
49Q10,		    
35J61,   		
49K20,   		
49K40,   	 	
74R05,   		
74B99.				

\newpage
\tableofcontents
\newpage

\section{Introduction}
Finding optimal shapes such that a physical system
exhibits an intended behaviour is of great interest for plenty of
engineering applications. For example design questions arise in 
the construction of
air- and spacecrafts, wind and combustion turbines, wave guides and inductor 
coils. 
More examples can be found in \cite{DelZol11} and references therein.
The physical system is usually modelled by a pde or a coupled pde system
supplemented with suitable boundary conditions.
In certain cases the state is given as a minimiser
of an energy, e.g., an equilibrium state of an elastic membrane,
which has to be in \red a\black~set of admissible states.
The solution is then characterised by a variational inequality
holding for test-functions on the sets of admissible states.

The treatment of optimal shape and control problems for variational inequalities
is substantially more difficult as without constraints, where
the sets of admissible states is a linear space.
For optimal control problems there exist a rapidly growing literature exploring
different types of stationarity conditions and their approximations
(see, for instance, \cite{hisu,MP84}).
However shape optimisation problems for systems described by
variational inequalities
are less explored and reveal additional difficulties due to
the intricated structure of the set of admissible domains.
Some results following the paradigm \emph{first optimise-then discretise}
can be found in \cite{JaruSoko03,neitsokozol88,MR691007,MR1055179,sokosolo,MR895743} and 
for the \emph{first discretise-then optimise} approach we refer to 
\cite{HO3,HO1,HO2}.

The main aim of this paper is to establish sensitivity estimates
and material derivatives for certain nonlinear elliptic variational inequalities
with respect to the domain. Our approach is based on the paradigm first optimise-then
discretise, thus the sensitivity is derived in an infinite dimensional setting. 
In order to
\red encapsulate \black
the main arguments needed in the proof
of \red the \black main results and to increase their applicability,
we investigate the optimisation problems firstly on an abstract operator
level formulated over a polyhedric subset $K$ of some reflexive Banach space $V$.
The domain-to-state map is there replaced by a parametrised family
of operators $(\C A_t)$ and sensitivity estimates are shown
in Theorem \ref{thm:energy_sensitivity} and Theorem \ref{thm:sens_At}
under general assumption (see Assumption (O1) and Assumption (O2)).
By \red strengthening\black~the assumptions (see Assumption (O3))
differentiability with respect to the parameter $t$ has been shown
in Theorem \ref{thm:mat_abstract}.
One crucial requirement is the polyhedricity of the closed convex set $K$
on which the operators are defined.
The results are applicable for optimal shape as well as for optimal control problems.

Equipped with the proven abstract results we resort to shape optimisation problems
where the state system is a variational inequality of semilinear elliptic type given by
$$
	u\in K_{\psi_\Omega}
	\;\text{ and }\;\forall \ph\in K_{\psi_\Omega}:\quad dE(\Omega,u;\ph-u)\geq 0
$$
with the energy
$$
	E(\Omega,u)
	=\int_\Omega \frac{1}{2}|\nabla u|^2+\frac{\lambda}{2}|u|^2+W_\Omega(x,u)\dx\quad(\lambda>0)
$$
and the upper obstacle set
$$
	K_{\psi_\Omega}
	=\big\{v\in H^1(\Omega):\,v\leq\psi_\Omega\text{ a.e. in }\Omega\big\}.
$$
In the classical theory
of VI-constrained shape optimisation problems
established in \cite{sokozol92},
linear variational inequalities with constant obstacle
and $W_\Omega(x,u)=f(x)$
for some given fixed function $f:D\to\R$ defined on a ``larger set'' $D\supset\Omega$
have been investigated by means of conical derivatives of projection operators in Hilbert spaces
\red as used\black~in \cite{Mi76}.
For results on topological sensitivity analysis for variational inequalities
and numerical implementations we refer to \cite{HL11} as well as \cite{MR3307830}.

In our paper we allow for semilinear terms in the variational inequality
by including convex contributions to $W_\Omega$ with respect to $u$
and also consider a dependence of $W_\Omega$ and $\psi_\Omega$ on $\Omega$ in a \red quite\black~general sense.
As presented in the last section of this work $\psi_\Omega$ may itself be a solution
of a variational inequality.
Such general $\Omega$-dependence of the obstacle will be referred to as
``dynamic obstacle'' in constrast to the case of a ``static obstacle''
where $\psi_\Omega(x)=g(x)$ for some fixed function $g:D\to\R$.

\red
On the one hand the results for VI-constrained shape optimisation problems
in \cite{sokozol92} are extended in the present contribution
to certain semi-linear cases, dynamic obstacles $\psi_\Omega$ and
dynamic potential functions $W_\Omega$.
On the other hand we establish these results by invoking
abstract sensitivity results for operators on Banach spaces (which we establish before)
and without reformulating the problems by means of projection operators
as done in \cite{sokozol92}.
One advantage of our different technique is that
we encapsulate the main arguments for obtaining material derivatives
in general theorems which are freed of concrete representation
of the (integral) operators.
The occurring operators are supposed to be uniformly monotone (see (O1) (iii) and (O2) (i))
-- a crucial assumption to gain sensitivity estimate in a general setting.
\black

To apply the abstract
results
\red
to the shape optimisation problem mentioned above
\black
we perform the transformation
$u\mapsto y:=u-\psi_{\Omega}$ such that the transformed problem
is formulated over the cone $H_-^1(\Omega)$, i.e., the non-positive
half space of $H^1(\Omega)$.
Existence of the material derivative $\dot y$ which turns out to be
the unique solution of a variational inequality
considered over the cone $T_y(H_-^1(\Omega))\cap kern(dE(u;\cdot))$
and strong convergence of the corresponding difference quotients
are established in Theorem \ref{thm:weakMatDer} and Corollary
\ref{cor:strongMatDer}.
The variational inequality characterising the material derivative $\dot u$
is then established in Corollary \ref{cor:dotU}.
Moreover in the case of a static obstacle and $H^2(\Omega)$-regularity for $u$
we derive relations for the state-shape derivative $u'$
in Theorem \ref{thm:shapeDer} and Corollary \ref{cor:stateShapeDer}.

The theorems for the semilinear case are then applied to
a specific model problem from continuum damage mechanics.
\red Here we \black consider an elastic solid which undergoes deformation
and damage processes in a small strain setting.
The state of damage is modelled by a phase field
variable $\chi$ which influences the material stiffness and which is described
by \red parabolic variational inequality\black~forcing the variable $\chi$
to be monotonically decreasing in time.
We consider a time-discretised version of the evolution system
(but we stay continuous in the spatial components)
where the damage variable fulfills for all time steps the constraints
$$
	\chi^N\leq\chi^{N-1}\leq\ldots\leq\chi^0\leq 1
	\text{ a.e. in }\Omega.
$$
Such constraints lead to $N$-coupled variational inequalities with
dynamic obstacle sets of the type
$$
	K^{k-1}(\Omega)
	=\big\{v\in H^1(\Omega):\,v\leq\chi^{k-1}\text{ a.e. in }\Omega\big\},
	\quad k=1,\ldots,N.
$$
Our objective is to find an optimal shape $\Omega$ such that
the associated displacement fields $(\ub^k)_{k=1}^N$ and damage phase fields
$(\chi^k)_{k=1}^N$
minimise a given tracking type cost functional.
We derive relations for the material derivative and establish necessary optimality
conditions for optimal shapes
which are summarised in Proposition \ref{prop:optimality}.

\paragraph*{Structure of the paper.}

In Section 2 we recall some basics notions from convex analysis.
\red For reader's convenience and for the sake of clarity
we derive in Appendix A tangential and normal cones 
of $K_{\psi_\Omega}$
and prove polyhedricity of $K_{\psi_\Omega}$ by invoking arguments
from \cite{Mi76, bosha, HP05}. \black

In Section 3 we establish sensitivity and material derivative results
in an abstract operator setting (see Theorem \ref{thm:energy_sensitivity},
Theorem \ref{thm:sens_At} and Theorem \ref{thm:mat_abstract}).
Some results are even applicable to quasi-linear problems such as to $p$-Laplace equations.
The advantage of this approach is that the theorems
can be applied to a large class of optimisation problems including
shape optimisation and optimal control problems.

This flexibility is demonstrated in Section 4 where semilinear VI-constrained
shape optimisation problems with an energy and obstacle of type
$E(\Omega,u)$ and $K_{\psi_\Omega}$ from above are treated.
By applying the abstract results from Section 3 we derive
sensitivity estimates for the shape-perturbed problem
in Proposition \ref{prop:At_mon_cont}, material derivatives
in Theorem \ref{thm:weakMatDer} and state-shape derivatives
in Theorem \ref{thm:shapeDer}.

Finally, in Section 5,
we \red invoke results \black from Section 4 \red in order to investigate \black
a \red shape optimisation \black problem \red from \black continuum damage mechanics
where dynamic obstacles \red arise\black.

\section{Notation and basic relations}
For the treatment of variational inequalities
we recall certain well-known cones from convex analysis
(the definitions can, for instance, be found in \cite[Chapter 2.2.4]{bosha} and \cite[Chapter 4.1]{sokozol92}).
Let $K\subseteq V$ be a subset of a real Banach space $V$
and denote by $V^*$ its topological dual space.

The \emph{radial cone}  at $y\in K$ of the set $K$ is defined by
\begin{equation}
	C_y(K) := \{ w\in V: \; \exists t>0, y+tw \in K  \},
	\label{clarke_cone}
\end{equation}
the \emph{tangent cone} at $y$ as
\begin{equation}
	T_y(K) := \overline{C_y(K)}^V
	\label{tangent_cone}
\end{equation}
and the \textit{normal cone} at $y$ as
\begin{equation}
	N_y(K) := \{w^*\in V^*:\; \forall v\in K,\,\langle w^*,v-y\rangle_V\leq 0\}.
	\label{normal_cone}
\end{equation}
Furthermore we introduce the polar cone of a set $K$ as
\begin{equation}
	[K]^\circ := \{w^*\in V^*:\; \forall v\in K,\,\langle w^*,v\rangle_V\leq 0\},
	\label{polar_cone}
\end{equation}
and the orthogonal complements of elements $y\in V$ and $y^*\in V^*$
\begin{align*}
	&[y]^\perp:=\{w^*\in V^*:\;\langle w^*,y\rangle_V=0\},\\
	&kern(y^*):=[y^*]^\perp:=\{w\in V:\;\langle y^*,w\rangle_V=0\}.
\end{align*}
The normal cone may also be written as
\begin{equation}
	N_y(K)=\left[T_y(K) \right]^\circ=\left[C_y(K) \right]^\circ.
	\label{normal_cone_repr}
\end{equation}
In combination with the bipolar theorem (see \cite[Prop. 2.40]{bosha}) we obtain
\begin{align}
	T_y(K)=[[T_y(K)]^\circ]^\circ=[N_y(K)]^\circ.
	\label{bipolar}
\end{align}
We recall that a closed convex set $K\subseteq V$ is
\textit{polyhedric} if (cf. \cite{hisu})
\begin{equation}
	\forall y\in K,\;\forall w\in N_y(K),\quad\overline{C_y(K)\cap [w]^\perp}^V=T_y(K)\cap [w]^\perp.
	\label{polyhedric_set}
\end{equation}
Note that the inclusion ``$\subseteq$'' is always satisfied above.
Due to Mazur's lemma and the convexity of the involved sets,
the closure in $V$ can also be taken in the weak topology.

The following lemma shows a useful implication of \eqref{polyhedric_set}
involving variational inequalities arising from (possibly non-)linear operators.
\begin{lemma}
\label{lemma:polyhedr}
	Let $K\subseteq V$ be a polyhedric subset of $V$.
	\begin{itemize}
		\item[(i)]
			Let $\C A:K\to V^*$ be an operator
			and let $y$ be a solution of the following variational inequality
			\begin{align}
				y\in K\quad\text{and}\quad\forall\ph\in K:\;
				\langle\C A(y),\ph-y\rangle_{V}\geq 0.
				\label{eq:abstrVI}
			\end{align}
			Then it holds
			\begin{align}
				\overline{ C_y(K)\cap \text{kern}(\mathcal A(y))}=T_y(K)\cap \text{kern}(\mathcal A(y)).
				\label{polyhedric_set_impl}
			\end{align}
		\item[(ii)]
			For all $v\in V$ it holds
			\begin{align*}
				\overline{ C_y(K)\cap [v-y]^\perp}=T_y(K)\cap [v-y]^\perp,
			\end{align*}
			where $y$ denotes the projection of $v$ on $K$.
	\end{itemize}
\end{lemma}
\begin{proof}
	\textit{To (i):}
	We infer from \eqref{eq:abstrVI} that
	$-\C A(y)\in N_y(K)$.
	Thus definition \eqref{polyhedric_set} implies
	$$
		\overline{ C_y(K)\cap \textit{kern}(-\mathcal A(y))}=T_y(K)\cap \textit{kern}(-\mathcal A(y)).
	$$
	The identity $\textit{kern}(-\C A(y))=\textit{kern}(\C A(y))$ completes
	the proof.\\
	\textit{To (ii):}
	This follows from $v-y\in N_y(K)$.
\end{proof}
Let us consider an important class of polyhedral subsets
which will be utilised in Section \ref{sec:semilinProb} where
semilinear obstacle problems are treated.
\red We fix a \black Lipschitz domain $\Omega\subseteq\R^d$.
Moreover let $\psi\in H^1(\Omega)$ be a given function.
We define the upper obstacle set as
\begin{align}
	K_\psi:=\{w\in H^1(\Omega):\;w\leq \psi\text{ a.e. in }\Omega\}.
	\label{eqn:Kexample}
\end{align}
The proofs of the following results are based on arguments from
\cite[\red Lemma \black 3.1-3.2, \red Theorem \black 3.2]{Mi76}
\red and are carried out in Appendix A.\black
\begin{theorem}
\label{thm:TN}
	Let $y\in K_\psi$ and $K_\psi$ be as in \eqref{eqn:Kexample}. Then it holds
	\begin{subequations}
	\begin{align}
		T_y(K_\psi)={}&\big\{u\in H^1(\Omega):\,\tilde u\leq 0\text{ q.e. on }\{\tilde y=\tilde\psi\}\big\},
			\label{tangSpace}\\
		N_y(K_\psi)={}&\big\{I\in H^1(\Omega)^*:\,I\in H^1(\Omega)_+^*\text{ and }\mu_I(\{\tilde y<\tilde\psi\})=0\big\},
			\label{normalSpace}
	\end{align}
	\end{subequations}
	where
	$\tilde y$ denotes a quasi-continuous representant of $y$ (the same for $\tilde u$ and $\tilde \psi$)
	and $\mu_I\in M_+(\ol\Omega)$ the measure associated to $I$ by Lemma \ref{lemma:identification}.
	
	Please notice that the sets
	\begin{align*}
		&\{\tilde y=\tilde\psi\}:=\{x\in\ol\Omega:\,\tilde y(x)=\tilde\psi(x)\},\\
		&\{\tilde y<\tilde\psi\}:=\{x\in\ol\Omega:\,\tilde y(x)<\tilde\psi(x)\}
	\end{align*}
	are calculated for arguments in $\ol\Omega$ (not only in $\Omega$).
\end{theorem}
\begin{theorem}[{cf. \cite[Th\'eor\`eme 3.2]{Mi76}}]
\label{theorem:Kpolyhedr}
	The set $K_\psi$ is polyhedric.
\end{theorem}

\section{Abstract sensitivity analysis}
\label{sec:abstrSense}
In this section we will derive
sensitivity estimates and relations for material derivatives
under general conditions.
We start in Section \ref{sec:sensMinEnergy}
with minimisers of certain $p$-coercive energy functionals
and deduce a H\"older-type estimate with exponent $1/p$.
We present an example which includes the quasi-linear
$p$-Laplacian $\Delta_p(\cdot)=\di(|\nabla \cdot|^{p-2}\nabla\cdot)$.
Then we proceed in Section \ref{sec:sensMonOp} with
solutions of monotone operators
where we are able to improve the estimates from Subsection \ref{sec:sensMinEnergy}.
For the case $p=2$ we even establish a Lipschitz type sensitivity estimate.
Finally in Subsection \ref{sec:varIneqMatDer}
we strengthen the assumptions in order to establish the weak material
derivative.
A crucial requirement will be the polyhedricity of the underlying set.

In this whole section $V$ will denote a Banach space,
$K\subseteq V$ a closed convex subset and $\tau>0$ a fixed constant.

\subsection{Sensitivity result for minimisers of energy functionals}
\label{sec:sensMinEnergy}
Our starting point is a family of energy functionals 
$$
	E:[0,\tau] \times V\rightarrow \R,
$$
where we denote the set of attained infima at $t\in[0,\tau]$ by 
\begin{align}
	\red U(t)\black := \big\{ u^t\in V: \inf_{\ph\in K} E(t,\ph) = E(t,u^t)  \big\}.
	\label{XSet}
\end{align}
Our aim is to establish a general result showing the 
convergence of minimisers of $E(t,\cdot)$ to minimisers 
 of $E(0,\cdot )$ as $t\searrow 0$.
Before we state our abstract sensitivity result,
we recall \cite[Theorem 1]{savare98} which will be used in a subsequent proof:
\begin{theorem}[{\cite[Theorem 1]{savare98}}]\label{thm:savare}
	Let $\red\triplenorm\black \cdot\red\triplenorm\black $ be a seminorm on $V$. Let $E:V\rightarrow \R$ be an energy functional such 
	that for all $v,w\in K$ the mapping 
	$s\mapsto\gamma(s):= E(sw+ (1-s)v))$ is $C^1$ on $[0,1]$. Let us denote by 
	$\mathcal A:K\rightarrow V^*$ the Gateaux-differential of $E$ which is supposed to 
	be $p$-coercive on $K$:
	$$
		\exists \alpha>0,\forall u,v\in K,
			\quad \langle\mathcal A(u)-\mathcal A(v), u-v\rangle_V \ge\alpha \red\triplenorm\black u-v\red\triplenorm\black ^p.
	$$
	Then every minimum $u$ of $E$ on $K$ satisfies:
	$$  \forall v\in K, \quad \frac{\alpha}{p} \red\triplenorm\black u-v\red\triplenorm\black ^p \le E(u) - E(v).$$
\end{theorem}
\subsubsection{\red H\"older-type estimate\black}
In what follows let $E$ satisfy the following assumption:
\vspace*{0.5em}\\
\textbf{Assumption (O1)}
\textit{
	Suppose that the energy functionals $E(t,\cdot)$ satisfies for a given
	$p\geq 1$:
	\begin{itemize}
		\item[(i)] 
			$\exists c_1>0,\exists c_2>0$, 
			$\forall \varphi\in K,$
			$E(\cdot,\ph)$ is differentiable and
			$$
				\forall t\in [0, \tau],\quad\partial_t E(t, \varphi) \le c_1 \|\varphi\|^p_V + c_2;
			$$
		\item[(ii)] $\exists c>0$, $\exists \Lambda>0$, 
			$\forall \varphi \in K$, $\forall t\in [0,\tau]$,
			$$
				E(t,\varphi) \ge c \|\varphi \|_V^p- \Lambda;
			$$
		\item[(iii)]
			$\forall t\in[0,\tau]$, $E(t,\cdot)$ is Gateaux-differentiable
			and
			$$
				\exists \alpha>0,\forall u,v\in K, \quad
				\langle\mathcal A_t(u)-\mathcal A_t(v), u-v\rangle_V
				\ge \alpha \red\triplenorm\black u-v\red\triplenorm\black ^p,
			$$
			where $\langle \mathcal A_t(v),w \rangle_V := dE(t,v;w)$
			and $\red\triplenorm\black \cdot\red\triplenorm\black $ is a semi-norm on $V$;
		\item[(iv)] $\forall v,w\in K$, $\forall t\in [0,\tau]$, 
			$$
				s\mapsto\gamma(s):= E(t,sv+(1-s)w)\text{ is }C^1([0,1])
			$$
	\end{itemize}
}
We are in the position to state and prove our sensitivity result:
\begin{theorem}\label{thm:energy_sensitivity}
	Let $E:[0,\tau]\times V \rightarrow \R$ be a family of energy functionals 
	satisfying Assumption (O1) and let $\red U(t)\black$ be non-empty for every $t\in[0,\tau]$.
	Then $\red U(t)\black=\{u^t\}$ is a singleton and 
 	there exists a constant $c>0$ such that for all 
	$t\in [0,\tau]$:
	$$
		\red\triplenorm\black u^t - u^0\red\triplenorm\black  \le ct^{1/p}.
	$$
\end{theorem}
\begin{proof}
	Let $t\in [0,\tau]$
	and $u^t\in \red U(t)\black$. 
  Let us first show that $u^t$ is bounded in $V$ uniformly in $t$.  According to 
	Assumption (O1) (i)-(ii), the definition of $u^t$  and the mean 
	value theorem, we obtain $\eta_t \in (0,t)$ such that
	\begin{equation}\label{eq:est_ut}
		\begin{split}
			c\| u^t\|_V^p -\Lambda & \le E(t,u^t)\\
	    	& \le E(t,u^0)  \\
		    & = E(0,u^0) + t \partial_t E(\eta_t, u^0) \\
				& \le E(0,u^0) + t\big(c_1 \| u^0\|_V^p + c_2\big).
		\end{split}
 	\end{equation}
	This shows that $\| u^t\|_V \le C$ for all $t\in [0,\tau]$ for some 
	constant $C>0$. 
	Furthermore applying Theorem~\ref{thm:savare} by using Assumption (O1) (iii)-(iv) shows
	\begin{align}
		c\red\triplenorm\black u^t-u^0\red\triplenorm\black ^p & \le E(t,u^t) - E(t, u^0),\\
		c\red\triplenorm\black u^t - u^0\red\triplenorm\black ^p & \le E(0,u^0) - E(0,u^t).
	\end{align}
	Adding both inequalities, applying the mean value theorem twice with some
	$\eta_t, \zeta_t \in (0,t)$ and 
	using Assumption (O1) (i) and the estimate \eqref{eq:est_ut} yields
	\begin{equation}
		\begin{split}
		2c\red\triplenorm\black u^t-u^0\red\triplenorm\black ^p & \le E(t,u^t) - E(t, u^0) + E(0,u^0) - E(0,u^t)\\
		& \le t\big( \partial_t E(\eta_t, u^t) - \partial_t E(\zeta_t, u^0)\big) \\
		& \le t C ( \|u^t\|_V^p + \|u^0\|_V^p)\\
		& \stackrel{\eqref{eq:est_ut}}{\le} t C (1 + \|u^0\|_V^p).
		\end{split}
	\end{equation}
This finishes the proof. 	
\end{proof}

\subsubsection{\red Example: $p$-Laplace equation\black}
\hspace*{0.1em}\\
As an application of Theorem \ref{thm:energy_sensitivity} let us consider
the $p$-Laplace equation
$$
	-\di(|\nabla u|^{p-2}\nabla u)=f\quad\text{in }K=V=\ac W_p^1(\Omega)
$$
on a bounded Lipschitz domain $\Omega$
and the associated energy given by
$$
	E(0,\ph)=
		\frac{1}{p}\int_\Omega |\nabla \varphi|^p\dx - \int_\Omega f \varphi \dx,
		\quad\varphi \in \ac W_p^{1}(\Omega).
$$
\red 
Here $\ac W_p^{1}(\Omega)$ denotes the closure of $C^\infty_c(\Omega)$ in the $\|\cdot \|_{W^1_p}$-norm.
The energy of the perturbed
problem transported to $\Omega$ is of the form\black
\begin{align*}
	E(t, \varphi) = \frac{1}{p}\int_\Omega \xi(t)|B(t)\nabla \varphi|^p - f(t) \varphi \dx.
\end{align*}
\red
More precisely this type of energy arises if one considers the energy on a
perturbed domain $\Phi_t(\Omega)$
and apply a change of variables, i.e.
\begin{align*}
	 \tilde E(t, \varphi) = \frac{1}{p}\int_{\Phi_t(\Omega)} |\nabla \varphi|^p-f\varphi \dx
		=\frac{1}{p}\int_{\Omega} \det(\partial\Phi_t)|(\partial\Phi_t)^{-\top}\nabla \tilde \varphi|^p-f\det(\partial\Phi_t)\tilde \varphi \dx,
\end{align*}
where $\tilde \varphi = \varphi \circ \Phi_t$. Now 
the minimisation of 
$\tilde E(t,\cdot)$ over $\ac W^1_p(\Phi_t(\Omega))$ is equivalent to the minimisation of 
$E(t,\varphi):=\tilde E(t,\varphi \circ\Phi^{-1})$ over  $ \ac W^1_p(\Omega)$.
\black

More generally we assume that $\xi:[0,\tau] \to \R$ and 
$B:[0,\tau] \to \R^{d\times d}$ are $C^1$-functions which satisfy $\xi(0) =1$ and $B(0) =I$. 
Moreover let $f(0)=f$ and $f(\cdot,x)$ be differentiable
and $f'(t)\in L_{p'}(\Omega)$ be uniformly bounded where $p'=p/(p-1)$ denotes 
the conjugate of $p$.
We check that the assumptions in (E) are satisfied:

Indeed, we have
$$
\partial_t E(t,\varphi) = \int_\Omega \xi'(t) \frac{1}{p}|B(t)\nabla \varphi|^p 
+  \xi(t) |B(t)\nabla \varphi|^{p-2} B(t)\nabla \varphi\cdot B'(t) \nabla \varphi \dx - \int_\Omega  f'(t) \varphi 
\dx.
$$
Thus applying H\"older and Young's inequalities we verify Assumption (O1) (i):
\begin{align*}
	\partial_t E(t, \varphi )
	\le{}& \int_\Omega \xi'(t) \frac{1}{p}|B(t)\nabla \varphi|^p +  \xi(t) |B(t)\nabla \varphi|^{p-2} B(t) \nabla \varphi \cdot B'(t)\nabla \varphi \dx  - \int_\Omega  f'(t) \varphi \dx \\
	\le{}&  c \int_\Omega |\nabla \varphi|^p + |f'(t)|| \varphi |\dx \\
				   \le{}& c   \|\nabla \varphi\|^p_{L_p} + 1/p'\|f'(t)\|_{L_{p'}}^{p'}
				   +\frac{1}{p}\|\varphi\|_{L_p}^p.  
\end{align*}
On the other hand using Young's and Poincar\'e's inequality
with small $\varepsilon>0$
\begin{align*}
	E(t,\varphi )
		\ge{}& c \|\nabla \varphi \|^p_{L_p}
			- 1/p'  (p\varepsilon)^{-\frac{1}{p-1}}\|f(t) \|_{L_{p/(p-1)}}^{p'}
			- \varepsilon\|\varphi\|_{L_p}^p\\
		\ge{}& c_1\|\varphi \|^p_{W_p^1}
			- c_2	- \varepsilon\|\varphi\|_{L_p}^p .
\end{align*}
Thus we have verified Assumption (O1) (ii).
Assumption (O1) (iii) follows from uniform $p$-monotonicity of $-\Delta_p(\cdot)$
and Assumption (O1) (iv) by direct calculations.

Finally we may use
Theorem \ref{thm:energy_sensitivity} and obtain $\| u^t - u\|_{W^1_p(\Omega)} \le 
ct^{1/p}$ for some constant $c>0$ and all sufficiently small $t>0$. In the case
of the usual Laplace equation, that is for $p=2$, we get
$\| u^t - u\|_{H^1(\Omega)} \le ct^{1/2}$.

\subsection{Sensitivity result for uniformly monotone operators}
\label{sec:sensMonOp}
In this section we develop sensitivity results for variational inequalites 
involving uniformly monotone operators.
Let $V$ be a normed space, $V^*$ its dual space and $K\subseteq V$ be a closed convex subset.

\subsubsection{\red Enhanced H\"older-type estimate\black}
\red The following assumptions are used in this section: \black

\noindent
\textbf{Assumption (O2)}
\textit{
	Suppose that  $ (\mathcal A_t):K \rightarrow V^*$, $t\in[0,\tau]$
	is a family of operators such that for a given $p\geq 1$: 
	\begin{itemize}	
		\item[(i)]  $\exists \alpha>0 $, $\forall t\in [0,\tau]$, $\forall u,v\in K$:
			$$  \alpha \| u-v \|_V^p \le \langle \mathcal A_t(u) - \mathcal A_t(v), u-v\rangle_V;   $$
		\item[(ii)] $\forall u\in K$, $\exists c>0 $, $\forall t\in [0,\tau]$, $\forall v\in K$, 
			$$	|\langle \mathcal A_t(u) -  \mathcal A_0(u), u - v \rangle_V| \le c t \|u-v\|_V. 	$$
	\end{itemize}
}

\begin{theorem}\label{thm:sens_At}
	Suppose that 
	$(\mathcal A_t): K \rightarrow V^*$ is a family of operators satisfying Assumption (O2). 
	For every $t >0$ we denote by $u^t\in K$ a solution of the 
	variational inequality
	\begin{equation}\label{eq:ut_abstract}
		u^t\in K\text{ and } \forall v\in K,\,\langle \mathcal A_t(u^t) , v - u^t \rangle_V \ge 0.
	\end{equation}
	Then there exists a $c>0$ such that
	$$ \forall t\in[0,\tau]:\quad  \| u^t - u^0\|_V \le c t^{\frac{1}{p-1}}.  $$
\end{theorem}
\begin{proof}
	Taking into account Assumption (O2) and \eqref{eq:ut_abstract}:
	\begin{align*}
   	\alpha \| u^t-u^0 \|_V^p & \le \langle \mathcal A_t(u^t) -  \mathcal A_t(u^0), u^t - u^0\rangle_V\\ 
    & \le  - \langle \mathcal A_t(u^0), u^t - u^0\rangle_V\\
		& = \langle \mathcal A_0(u^0), u^t - u^0\rangle_V
			+ \langle \mathcal A_0(u^0) - \mathcal A_t(u^0), u^t - u^0\rangle_V\\
		& \le |\langle \mathcal A_0(u^0) - \mathcal A_t(u^0), u^t - u^0\rangle_V|\\
		& \le ct\|u^t - u^0\|_V. 
	\end{align*}
\end{proof}
\begin{remark}
	In the important case $p=2$ Theorem \ref{thm:sens_At} yields a Lipschitz type estimates.
\end{remark}

\subsubsection{\red Example: $p$-Laplace equation\black}
It can be checked that the $p$-Laplace example from Subsection \ref{sec:sensMinEnergy}
where $\C A_t$ is given by
$$
	\langle\C A_t(u),\ph\rangle_{\ac W_p^1}
	=\int_\Omega \xi(t)|B(t)\nabla u|^{p-2}B(t)\nabla u\cdot B(t)\nabla\ph-f(t)\ph\dx
$$
also fulfills Assumption (O2).
Thus in this case Theorem \ref{thm:sens_At} gives a sharper estimate
than Theorem \ref{thm:energy_sensitivity}.

\subsection{Variational inequality for the material derivative}
\label{sec:varIneqMatDer}
In the previous section we have shown that under certain conditions on
$(\mathcal A_t)$ satisfied for $p=2$ the quotient $(u^t-u^0)/t$ stays bounded.
In this subsection we additionally assume that
$V$ is reflexive and
that $K\subseteq V$ is a polyhedric subset.
Then there will be a weakly converging subsequence of $(u^t-u^0)/t$
converging to some $z\in V$. If this $z$ is unique the whole sequence
converges and additionally satisfies some limiting equation
which is the subject of this subsection.

Let $(\C A_t)$ be as in Subsection \ref{sec:sensMonOp} and
define in accordance with \eqref{XSet} for all $t\in [0,\tau]$ the solution set
of the associated variational inequality as
\begin{equation}\label{eq:states}
	\red U(t)\black:= \big\{ u^t\in K: \; \forall \varphi \in K,\,\langle \mathcal A_t(u^t), \varphi - u^t \rangle 
		\ge 0\big\}.
\end{equation}
We will write $u:=u^0$ and $\C A:=\C A_0$.

The variational inequality for the material derivative will be deduced
from the following assumptions:
\vspace*{0.5em}\\
\noindent\textbf{Assumption (O3)}
\textit{
	Suppose that the family $(\C A_t)$ satisfies
	\begin{itemize}
		\item[(i)] for all $v,w\in V$ and all $u\in K$,
			$$ \langle \partial \mathcal A(u)w,v\rangle_V :=  \lim_{t\searrow 0} \left\langle \frac{\mathcal A(u+tw) - \mathcal A(u)}{t}, v\right\rangle_V  $$
		        and 
			$$ \langle \mathcal A'(u),v\rangle_V := \lim_{t\searrow 0} \left \langle \frac{\mathcal A_t(u) - \mathcal A(u)}{t}, v \right \rangle_V  $$		exist;
		\item[(ii)]  for all null-sequences $(t_n)$,
			for all sequences  $(v_n)$ in $V$ converging weakly to some $v\in V$,
			for all $u^{t_n} \in \red U(t_n)\black$ converging strongly to some $u\in K$, we have 
			$$
				\langle \mathcal A'(u), 	v \rangle_V
				 =  \lim_{n \rightarrow 0 } \left \langle \frac{\mathcal A_{t_n}(u^{t_n}) - \mathcal A(u^{t_n})}{t_n},  v_n \right\rangle_V;
	    $$
    \item[(iii)]
			\red
			there exists a null-sequence $(t_n)$
			\black
			such that 
			$u^{t_n}\in \red U(t_n)\black$ converges strongly to $u\in K$ and 
			$(u_n-u)/t_n$ converges weakly to some $z\in V$ and
			$$ \langle \partial\mathcal A(u) z, z \rangle_V
				\le \liminf_{n \rightarrow 0 } \left \langle \frac{\mathcal A(u^{t_n}) - \mathcal A(u)}{t_n}, \frac{u^{t_n}-u}{t_n} \right\rangle_V
			$$ 
			and for all $(v_n)$ in $V$ converging strongly to $v\in V$:
      $$
      	\langle \partial\mathcal A(u) z, v \rangle_V  = \lim_{n \rightarrow 0 } \left \langle \frac{\mathcal A(u^{t_n}) - \mathcal A(u)}{t_n},v_n \right\rangle_V.
			$$
	\end{itemize}
}

\begin{theorem}\label{thm:mat_abstract}
	Let $V$ be a reflexive Banach space and $K\subseteq V$ a polyhedric subset.
	Suppose that $\mathcal A_t:K \rightarrow V^*$, $t\in[0,\tau]$ is a family of operators 
	satisfying Assumption (O2) for $p=2$ and (O3). Suppose that $u^t\in \red U(t)\black$, i.e., $u^t$ solves
	\begin{align}
		u^t\in K, \quad \langle \mathcal A_t(u^t), \varphi - u^t \rangle_V \ge 0 \quad \forall \varphi \in K.
		\label{AtVIAss}
	\end{align}
	Then the material derivative $\dot u:=\mathrm{weak-lim}_{t\searrow 0}(u^t-u)/t$ exists and solves
	\begin{subequations}
	\begin{align}
			&\dot u\in T_u(K)\cap \text{kern}(\mathcal A(u))\text{ and }
				\label{dotUInc}\\
			&\forall \varphi\in T_u(K)\cap \text{kern}(\mathcal A(u)):\;\langle \partial \mathcal A(u)\dot u, \varphi - \dot u \rangle_V
			\ge - \langle \mathcal A'(u), \varphi - \dot u\rangle_V.
				\label{AtVI}
	\end{align}
	\end{subequations}
\end{theorem}
\begin{proof}
Let us firstly show \eqref{dotUInc}.
	We get by \eqref{AtVIAss}
  \begin{align}
  	&\forall \varphi \in \Kc(\Omega):\quad\langle\C A_t(u^t),\varphi - u^t\rangle\ge 0,
  		\label{eq:dEt}\\
		&\forall \varphi \in \Kc(\Omega):\quad\langle\C A(u),\varphi - u\rangle  \ge 0.
			\label{eq:dE0}
	\end{align}
	Thus testing \eqref{eq:dEt} with $u$ and \eqref{eq:dE0} with $u^t$ and 
	dividing by $t>0$, we obtain by setting $z^t:=(u^t - u)/t$
	\begin{equation}\label{eq:dE_dEt}
		\begin{split}
			\langle\C A_t(u^t),z^t\rangle\le 0, \qquad 
			\langle\C A(u),z^t\rangle\ge 0. 
   	\end{split}
	\end{equation}
	By invoking Theorem \ref{thm:sens_At} with $p=2$ we know that
	$u^t\to u$ strongly in $V$ and that
	$z^t$ is bounded in $V$ which allows us to choose a weakly convergence
	subsequence with limit $\dot u\in V$. We find (by omitting the subscript)
	\begin{align*}
		&\langle\C A_t(u^t),z^t\rangle
			-\langle\C A(u),\dot u\rangle\\
		&=\underbrace{\langle\C A_t(u^t)-\C A(u^t),z^t\rangle}_{\to0\text{ by Assumption (O3) (ii)}}
		+\underbrace{\Big\langle\frac{\C A(u^t)-\C A(u)}{t},u^t-u\Big\rangle}_{\to0\text{ by Assumption (O3) (iii)}}
		+\underbrace{\langle\C A(u),z^t-\dot u\rangle}_{\to 0}
	\end{align*}
	Therefore passing to the limit in \eqref{eq:dE_dEt}
	gives
	$ 0 \le \langle\C A(u),\dot u\rangle \le 0$ and thus $\dot u\in \text{kern}(\mathcal A(u))$.
	Furthermore we know by the definition of the radial cone that $z^t\in C_u(K)$.
	Taking the weak convergence $z^t\weaklim \dot u$ in $V$ and Mazur's Lemma into account we find $\dot u\in T_u(K)$.
	Thus \eqref{dotUInc} is proven.

Now we will show \eqref{AtVI} by using \eqref{AtVIAss} and obtain for every $\varphi \in V$:
\begin{equation}\label{eq:limit_abstract}
	\begin{split}
	\langle \mathcal A(u^t) - \mathcal A(u), \varphi - u^t \rangle  ={}&  
	\langle \mathcal A(u^t) - \mathcal A_t(u^t), \varphi - u^t \rangle
		+ \langle \mathcal A_t(u^t) - \mathcal A(u), \varphi - u^t \rangle\\
	    \ge{}& \langle \mathcal A(u^t) - \mathcal A_t(u^t), \varphi - u^t 
	    \rangle - \langle\mathcal A(u), \varphi - u^t \rangle. 
        \end{split}
\end{equation}
By definition of the radial cone $C_u(K)$ (see \eqref{clarke_cone})
we find for every $\varphi \in C_u(K)$ 
a $t^*>0$ such that for all $t\in [0,t^*]$: $u + t \varphi \in K$.
Plugging this test-function into \eqref{eq:limit_abstract} we obtain
for all $\varphi \in C_u(K)$
\begin{equation}\label{eq:limit_abstract-1}
	\begin{split}
		\langle \mathcal A(u^t) - \mathcal A(u), t\varphi - (u^t-u) \rangle
			\ge \langle \mathcal A(u^t) - \mathcal A_t(u^t), t\varphi - (u^t-u) \rangle
				- \langle\mathcal A(u), t\varphi - (u^t-u) \rangle.
   \end{split}
\end{equation}
Dividing the previous equation by $t^2$ and setting $z^t := (u^t-u)/t$,
we obtain 
\begin{equation}\label{eq:limit_abstract-2}
	\begin{split}
		\left\langle \frac{\mathcal A(u^t) - \mathcal A(u)}{t}, 
		\varphi - z^t \right\rangle  \ge - \left\langle 
		\frac{\mathcal A_t(u^t) - \mathcal A(u^t)}{t}, \varphi - z^t 
		\right\rangle - \frac{1}{t}\langle\mathcal A(u), \varphi - z^t \rangle. 
        \end{split}
\end{equation}
Now let $\varphi\in C_u(K)\cap kern(\C A(u))$.
Then because of $\langle\C A(u),\ph\rangle=0$ and the
definition of $u\in \red U(0)\black$ (testing the relation in \eqref{eq:states} with $u^t$),
we find
$$
	- \langle\mathcal A(u), \varphi - z^t \rangle \ge 0.
$$
Thus \eqref{eq:limit_abstract-2} reads
\begin{equation}\label{eq:limit_abstract-3}
	\begin{split}
		\left\langle \frac{\mathcal A(u^t) - \mathcal A(u)}{t}, 
		\varphi - z^t \right\rangle  \ge - \left\langle 
		\frac{\mathcal A_t(u^t) - \mathcal A(u^t)}{t}, \varphi - z^t 
		\right\rangle. 
  \end{split}
\end{equation}
Using Assumption (O3) we may take the lim$\,$sup on both sides to obtain
(note that $-\limsup(...)=\liminf-(...)$)
$$
	\langle \partial \mathcal A (u)z, \varphi - z \rangle \ge - \langle \mathcal 
	A'(u), \varphi - z\rangle \quad \forall \varphi\in C_u(K)\cap 
	kern(\mathcal A(u)).
$$
Via density arguments we obtain the inequality
for all $\varphi\in \ol{C_u(K)\cap kern(\mathcal A(u))}$.
Finally using polyhedricity of $K$ and Lemma \ref{lemma:polyhedr} (i)
finish the proof.
\end{proof}

\section{A semilinear dynamic obstacle problem}
\label{sec:semilinProb}
In this section we are going to apply the theorems from Section \ref{sec:abstrSense}
to generalised obstacle problems with convex energies.
\red We \black present a generalised obstacle problem. It also covers
previous results from \cite{sokozol92} where the zero obstacle case has been
treated as a special case.
A non-trivial example from continuum damage mechanics is presented afterward
in Section \ref{sec:damage}.

\subsection{\red State equation}
Let $D\subseteq \R^d$ be an open and bounded subset.
We consider a convex energy of the following type
\begin{equation}\label{eq:energy}
	E(\Omega, \varphi) := \int_\Omega \frac{1}{2}|\nabla \varphi|^2 + 
	\frac{\lambda}{2}|\varphi|^2 + W_{\Omega}(x,\varphi)\dx,\quad
	\varphi\in H^1(\Omega),
\end{equation}
where $\Omega\subseteq D $ is a bounded Lipschitz domain and
$\lambda>0$.
The energy is minimised over the convex set
\begin{equation*}
	\Kc_{\psi_\Omega}(\Omega) :=\big\{ \varphi\in H^1(\Omega):\, \varphi \le \psi_\Omega \text{ a.e. in } \Omega\big\}.
\end{equation*}
A particularity of this setting is that, besides the density function $W_\Omega$, 
also the obstacle
function $\psi_\Omega$ is allowed to depend on the shape variable $\Omega$
(the precise assumptions are stated below in Assumption (A1)):
\begin{align*}
	&\textit{dynamic density function:}
		&&\Omega\mapsto W_\Omega\quad\\
	&\textit{dynamic obstacle:}
		&&\Omega\mapsto \psi_\Omega\in H^1(\Omega)
\end{align*}
In the special case $\psi_\Omega\equiv0$ we write
$\Kc(\Omega) := \Kc_0(\Omega)$. 
\begin{remark}
\label{rmk:setting}
\begin{itemize}
	\item[(i)]
		An important class which is covered by our setting
		are \textit{static obstacle problems}
		where $\psi_\Omega:=\Psi|_\Omega$
		with a given function $\Psi\in H^2(D)$.
	\item[(ii)]
		The energy $E(\Omega,\cdot)$ is motivated 
		by time-discretised parabolic problems,
		where an additional $\lambda$-convex non-linearity may be included in $E$.
		By choosing a small time step size, the
		incremental minimisation problem may take the form \eqref{eq:energy}.

		\red In the context of \black time-discretised damage models
		in Section \ref{sec:damage}
		we are faced with iterative obstacle problems.
		In this case the obstacle $\psi_\Omega$
		itself is a solution of a variational inequality
		describing the damage profile from the previous time step.
		As we will see it suffices to have $H^1(\Omega)$-regularity
		of the damage profile
		provided that the material derivative of the obstacle
		exists in $H^1(\Omega)$ and the initial value is in $H^2(\Omega)$.
		We will present this application in the last section.
\end{itemize}
\end{remark}
For later use we recall that the Sobolev exponent $2^*$ depending
on the spatial dimension $d$
to the space $H^1(\Omega)$ is defined as
\begin{align}
	&2^*:=
	\begin{cases}
		\frac{2d}{d-2}
		&\text{if }d>2,\\
		\textit{arbitrary in }[1,+\infty)
		&\text{if }d=2,\\
		+\infty
		&\text{if }d=1.
	\end{cases}
\end{align}
Its conjugate $(2^*)'$ is given by $\frac{2^*}{2^*-1}$ with the
convention that $(2^*)':=1$ for $2^*=+\infty$.
For well-posedness of the state system we require the following assumptions
(note that we restrict ourselves to the convex case which will be exploited
in the next sections):
\vspace*{0.5em}\\
\textbf{Assumption (A1)}
\textit{
	For all Lipschitz domains $\Omega\subseteq D$ it holds:
	\begin{itemize}
		\item[(i)]
   		$W_{\Omega}(x,\cdot)$ is convex and in $C^1(\R)$
   		for all $x\in\Omega$;
		\item[(ii)]
			the following map $H^1(\Omega)\to\R$ is assumed to be continuous
			(in particular the integral exists)
			$$
				y\mapsto \int_\Omega W_{\Omega}(x,y(x))\dx
			$$
			and bounded from below by
			$$
 				\int_\Omega W_{\Omega}(x,y(x))\dx\geq -c(\|y\|_{H^1}+1);
 			$$
    \item[(iii)]
    	for all $y,\ph\in H^1(\Omega)$:
			\begin{align*}
    		\int_\Omega\frac{ W_{\Omega}(x,y+t\varphi)- W_{\Omega}(x,y)}{t}\dx
 		   		\to{}& \int_\Omega\partial_y W_{\Omega}(x,y)\ph\dx
    		&&\text{ as }t\searrow 0
			\end{align*}
			(in particular the integral on the right-hand side exists);
    \item[(iv)]
    	$\psi_\Omega\in H^1(\Omega)$.
	\end{itemize}
}
\begin{remark}
\label{rmk:A1}
			Assumption (A1) (iii) and the continuity property from (A1) (ii) are satisfied if, e.g.,
			the following growth condition holds:
			There exist constants $\epsilon,C>0$ and functions
			$s\in L_1(\Omega)$ and $r\in L_{(2^*)'}(\Omega)$
			such that for all $x\in \Omega$ and $y\in \R$:
			\begin{align*}
				|W_{\Omega}(x,y)|\leq{}& C|y|^{2^*-\epsilon}+s(x),\\
				|\partial_y  W_{\Omega}(x,y)|\leq{}& C|y|^{2^*-1}+r(x).
			\end{align*}
\end{remark}
The assumptions in (A1)
in combination with the direct method in the calculus of variations
imply unique solvability of the
variational inequality fulfilled by the minimisers of $E(\Omega,\cdot)$.
\begin{lemma}
Under Assumption (A1) the energy \eqref{eq:energy} admits for each 
Lipschitz domain $\Omega \subseteq D$ a unique minimum $u$ (depending on $\Omega$)
on $\Kc_\psi(\Omega)$ which 
is given as the unique solution of
\begin{equation} \label{eq:VI_weak}
	\left\{
	\begin{aligned}
		&u\in \Kc_{\psi_\Omega}(\Omega)\text{ and }\forall \ph\in \Kc_{\psi_\Omega}(\Omega):\\
		&\int_\Omega \nabla u \cdot \nabla (\varphi - u)  + \lambda u (\varphi- u)
			+ w_{\Omega}(x,u)(\varphi-u)\dx\ge 0,
	\end{aligned}
	\right.
\end{equation}
where
$$
	w_{\Omega}(x,y) := \partial_y W_{\Omega}(x,y).
$$
\end{lemma}
%
In the sequel we will treat the variational inequality \eqref{eq:VI_weak}
by making use of the transformation for the state variable and its test-function:
$$
	y := u-\psi_\Omega\text{ and }\red \hat \varphi \black:= \varphi - \psi_\Omega.
$$
The variation inequality becomes a problem involving the standard obstacle set
$$
\Kc(\Omega) := \red K_0(\Omega)=\black\big\{ \varphi \in H^1(\Omega):\; \varphi\le 0 \text{ a.e. on  } \Omega\big\}.
$$
Substituting above tranformation into \eqref{eq:VI_weak} we 
obtain the following variational inequality:
\begin{align}
\label{eq:VI_weak2}
	\left\{
	\begin{aligned}
	&y\in \Kc(\Omega)\text{ and }\forall \ph\in \Kc(\Omega):\\
	&\int_\Omega \nabla y \cdot \nabla (\varphi - y)  + \lambda y (\varphi- 
	y) +   w_{\Omega}(x,y+\psi_\Omega)(\varphi-y)\dx\\
	&\qquad\ge - \int_\Omega \nabla \psi_\Omega \cdot \nabla (\varphi-y)
		+ \lambda \psi_\Omega(\varphi - y)\dx
	\end{aligned}
	\right.
\end{align}
Hence it will suffice to investigate the solution $y$ to deduce properties 
of the function $u$. 

\subsection{\red Perturbed problem\black}

	In this subsection we prove a shape sensitivity result
	for the variational inequality \eqref{eq:VI_weak2}.
	In what follows let us denote by $\Phi_t$ the flow generated by a vector 
	field $X\in C^1_c(D, \R^d)$. For  $\Omega \subseteq D$ denote by 
	$\Omega_t:=\Phi_t(\Omega)$, $t\ge 0$, the perturbed domains
	(see Appendix B for more details).
	
	The solution $y_t\in H^1(\Omega_t)$ to the perturbed
	variational inequality to \eqref{eq:VI_weak2} satisfies
\begin{align}
\label{eq:per_yt}
\left\{
\begin{aligned}
	&y_t\in \Kc(\Omega_t)\text{ and }\forall \ph\in \Kc(\Omega_t):\\
	&\int_{\Omega_t} \nabla y_t \cdot \nabla (\ph - y_t)
		+ \lambda y_t(\ph-y_t) +   w_{\Omega_t}(x,y_t+\psi_{\Omega_t})(\ph-y_t)\dx\\
		&\qquad
		\ge - \int_{\Omega_t} 
		\nabla \psi_{\Omega_t} \cdot \nabla (\ph-y_t) + \lambda \psi_{\Omega_t} (\ph - y_t )\dx.
\end{aligned}
\right.
\end{align}
We will sometimes write $y_t(X)=y_t$ to emphasise the dependence on $X$.
Please note that in general $y_0(X) = y_t(X)$ for all $t \ge 0$ and 
for all vector fields $X\in C^1_c(D, \R^2)$ with the property 
$ X\cdot n = 0 \text{ on } \partial \Omega $.
This implication will be used in the forthcoming Lemma \ref{lem:state_shape}.
Throughout this work we will adopt the following abbreviations:
\begin{align}
\begin{aligned}
	w^t_{X}(x,\ph) :={}& w_{\Omega_t}(\Phi_t(x),\ph),
	& W^t_{X}(x,\ph) :={}&  W_{\Omega_t}(\Phi_t(x),\ph),
	&\psi_X^t:={}&\psi_{\Omega_t}\circ \Phi_t,\\
	A(t):={}&\xi(t)(\partial \Phi_t)^{-1}(\partial \Phi_t)^{-T},
	&\xi(t):={}&\det{\partial \Phi_t},
	&y^t :={}&y_t \circ \Phi_t
\end{aligned}
\label{abbrev}
\end{align}
and (for $t=0$)
\begin{align*}
\begin{aligned}
	&\psi(x):=\psi_\Omega(x),
	&&w(x,\ph):=w^0_{X}(x,\ph).
\end{aligned}
\end{align*}
From Lemma~\ref{lemma:phit}
we can directly infer the following convergences and estimates
\begin{lemma}\label{lem:convergence_comp}
	Let $X\in C^1_c(D, \R^d)$ be given. Then it holds:
	\begin{itemize}
		\item[(i)]
			the convergences as $t\searrow 0$:
			\begin{subequations}
			\label{AxiDer}
			\begin{align}
			  \frac{A(t) - I}{t} \rightarrow{}& A'(0) = \di(X)I-\partial X - (\partial X)^\top && \text{ strongly in } C(\overline D, \R^{d,d}),\\
			  \frac{\xi(t)-1}{t} \rightarrow{}& \xi'(0) = \di(X) && \text{ strongly 
			  in } C(\overline D);
			\end{align}
			\end{subequations}
		\item[(ii)]
			there is a constant $t^*>0$ such that
			\begin{align*}
				\forall t \in & [0,t^*], \forall x\in \overline D, \forall\zeta\in\R^d,\quad A(t,x)\zeta\cdot \zeta \ge 1/2|\zeta|^2,\\
				\forall t\in & [0,t^*], \forall x\in \overline D, \quad \xi(t,x) \ge  1/2.
			\end{align*}
	\end{itemize}
\end{lemma}

Performing a change of variables and using 
$(\nabla y)\circ \Phi_t = (\partial \Phi_t)^{-T}\nabla (y\circ \Phi_t )$
it is easy to check that 
the transported function $y^t$ (which is defined on $\Omega$)
satisfies the relation
\begin{align}\label{eq:VI_weak_p_1}
	\left\{
	\begin{aligned}
	&y^t\in \Kc(\Omega)\text{ and }\forall \ph\in \Kc(\Omega):\\
	&\int_\Omega A(t)\nabla y^t \cdot \nabla (\ph - y^t)  + 
		\xi(t)\lambda y^t(\ph-y^t) +  \xi(t) w^t_{X}(x,y^t+\psi_X^t)(\ph-y^t) \dx \\
	&\qquad\ge\int_\Omega 
		-A(t)\nabla\psi_X^t \cdot \nabla (\ph-y^t) -\xi(t)\lambda\psi_X^t(\ph-y^t) \dx.
	\end{aligned}
	\right.
\end{align}
For later usage let us introduce the bilinear form
\begin{align*}
	\af^t(v,\red z\black) & :=  \int_\Omega A(t)\nabla v \cdot \nabla \red z\black + \xi(t)\lambda v\red z\black \dx,
\end{align*}
the operator $\mathcal A_t: K_{\psi_\Omega}(\Omega) \rightarrow H^1(\Omega)^*$
by
\begin{align}
	\langle\mathcal A_t(v),\red z\black \rangle_{H^1(\Omega)}
		:={}&\af^t(v,\red z\black)+\int_\Omega \xi(t) w^t_{X}(x,v)\red z\black\dx
	\label{Adef}
\end{align}
and the ``shifted'' operator $\tA_t: \Kc(\Omega) \rightarrow H^1(\Omega)^*$ by
\begin{align}
	\tA_t(v):=\C A_t(v+\psi_X^t).
	\label{tildeAdef}
\end{align}
By making use of this notation the variational inequality \eqref{eq:VI_weak_p_1} can be recasted as
\begin{equation}
	y^t\in\Kc(\Omega)\quad\text{and}\quad\langle\tA_t(y^t),\ph - y^t\rangle_{H^1}\ge 0\quad \text{ for all }\ph \in \Kc(\Omega).
	\label{eq:energy_compact}
\end{equation}
In the following we also write $\C A:=\C A_0$ and $\tA:=\tA_0$.
\subsection{\red Sensitivity estimate\black}
Our goal is to apply Theorem \ref{thm:sens_At} designed for abstract operators.
For this reason we make the following assumption in addition to (A1):\vspace*{0.5em}\\
\textbf{Assumption (A2)}
\textit{
\begin{itemize}
	\item[(i)]  $\forall X\in C^1_c(D, \R^d), \exists c>0,\forall t\in[0,\tau], \forall \chi \in H^1(\Omega)$, 
		$$ \|w^t_{X}(\cdot, \chi) - w(\cdot, \chi)\|_{L_{(2^*)'}(\Omega)} \le ct; $$
	\item[(ii)] $\forall X\in C^1_c(D, \R^d), \exists c>0,\forall t\in[0,\tau], \forall \chi_1, \chi_2\in H^1(\Omega)$, 
		$$ \|w^t_{X}(\cdot, \chi_1) - w^t_{X}(\cdot, \chi_2)\|_{L_{(2^*)'}(\Omega)} \le c \|\chi_1- \chi_2\|_{H^1(\Omega)};$$
	\item[(iii)]
		$\forall X\in C^1_c(D, \R^d), \exists c>0,\forall t\in[0,\tau],$
		$$
			\|\psi_X^t-\psi\|_{H^1(\Omega)}\leq ct.
		$$
\end{itemize}
}
We are now in the position to prove the following sensitivity result:
\begin{proposition}\label{prop:At_mon_cont}
	Let the Assumptions (A1)-(A2) be satisfied.
	Then the family of operators $(\tA_t)$ defined by \eqref{tildeAdef} fulfills
        \begin{itemize}	
		\item[(i)] $\exists \alpha>0 $, $\exists t^*>0$, $\forall t\in [0,t^*]$, $\forall v,\red z\black\in \Kc(\Omega)$,
		     \begin{equation}\label{eq:monotonicity_At}
			   	\alpha \| v-\red z\black \|_{H^1(\Omega)}^2 \le \langle \tA_t(v) - \tA_t(\red z\black), v-\red z\black\rangle;
			   \end{equation}
		\item[(ii)] $\forall v\in \Kc(\Omega)$, $\exists c>0 $, $\exists t^* >0$, $\forall t\in [0,t^*]$, $\forall \red z\black\in \Kc(\Omega)$,
			\begin{equation}\label{eq:cont_At}
				|\langle \tA_t(v) - \tA(v), v - \red z\black \rangle| \le c t \|v-\red z\black\|_{H^1(\Omega)}.
      \end{equation}
	\end{itemize}

\end{proposition}
\begin{proof}
	\textit{To (i):} We first show the monotonicity estimate  \eqref{eq:monotonicity_At}.
      	With the help of Lemma~\ref{lem:convergence_comp} (ii) and monotonicity of $w_{X}^t$ in
	the second variable (see Assumption (A1) (i))
	we obtain for all $v,\red z\black \in H^1(\Omega)$ and all small $t\ge0$
	\begin{equation}\label{eq:mon-1}
		\begin{split}
			&\frac12 \int_\Omega  |\nabla  (v-\red z\black)|^2 + \lambda |v-\red z\black|^2 \dx\\
			&\qquad\le \af^t(v-\red z\black,v-\red z\black)\\
			&\qquad\quad + \int_\Omega \xi(t)\big(w^t_{X}(x,v+\psi_X^t) - w^t_{X}(x,\red z\black+\psi_X^t)\big)\big((v+\psi_X^t)-(\red z\black+\psi_X^t)\big)\dx 
	  \end{split}	  
	  \end{equation}
	Thus \eqref{eq:monotonicity_At} is shown.
	
	\textit{To (ii):} Let us fix $v\in H^1(\Omega)$.
	Then by applying H\"older inequality, Sobolev embeddings and the assumptions in (A2)
	we find for all $\red z\black\in H^1(\Omega)$
	\begin{align*}
		&\langle \tA_t(v) - \tA(v), v - \red z\black \rangle\\
		&\qquad\le
			\underbrace{\int_\Omega (A(t)-I)\nabla v \cdot \nabla (v - \red z\black)\dx}_{\leq\|A(t)-I\|_{L_\infty}\|\nabla v\|_{L_2}\|\nabla(v-\red z\black)\|_{L_2}}
			+\underbrace{\int_\Omega\lambda (\xi(t)-1) v( v -  \red z\black) \dx}_{\leq \lambda\|\xi(t)-1\|_{L_\infty}\|v\|_{L_2}\|v-\red z\black\|_{L_2}}\\
		&\qquad\quad+\hspace*{-5.9em}\underbrace{\int_\Omega (A(t)\nabla\psi_X^t-\nabla\psi) \cdot \nabla (v-\red z\black) + 
			\lambda(\xi(t)\psi_X^t-\psi)(v-\red z\black)\dx}_{\qquad\qquad\qquad
				\leq\big(\|A(t)-I\|_{L_\infty}\|\nabla\psi_X^t\|_{L_2}+\|\nabla\psi_X^t-\nabla\psi\|_{L_2}
					+ \lambda\|\xi(t)-1\|_{L^\infty}\|\psi_X^t\|_{L^2}+\lambda\|\psi_X^t-\psi\|_{L_2}\big)\|v-\red z\black\|_{H^1}} \\
		&\qquad\quad+\underbrace{\int_\Omega (\xi(t)-1)w^t_{X}(x,v+\psi_X^t)(v-\red z\black)\, dx}_{\leq \|\xi(t)-1\|_{L_\infty}\|w^t_{X}(x,v+\psi_X^t)\|_{L_{(2^*)'}}
			\|v-\red z\black\|_{H^1}}\\
		&\qquad\quad+\hspace*{-4.4em}\underbrace{\int_\Omega (w^t_{X}(x,v+\psi_X^t) - w_{X}^t(x,v+\psi))(v - \red z\black)\dx.}_{\qquad\qquad\quad
			\leq \|w^t_{X}(x,v+\psi_X^t) - w^t(x,v+\psi)\|_{L_{(2^*)'}}\|v - \red z\black\|_{H^1}\,\leq\, \|\psi_X^t-\psi \|_{H^1} \|v - \red z\black\|_{H^1}}\\ 
		&\qquad\quad+\hspace*{-1.5em}\underbrace{\int_\Omega (w^t_{X}(x,v+\psi) - w(x,v+\psi))(v - \red z\black)\dx.}_{\qquad
			\leq \|w^t_{X}(x,v+\psi) - w(x,v+\psi)\|_{L_{(2^*)'}}\|v - \red z\black\|_{H^1}\,\le\,ct \|v-\red z\black\|_{H^1}} 
	\end{align*}
	Taking Lemma~\ref{lem:convergence_comp} into account and using Young's inequality,
	we obtain the assertion.
\end{proof}
The desired Lipschitz estimate immediately follows from Theorem~\ref{thm:sens_At}
since Proposition~\ref{prop:At_mon_cont} proves that Assumption (O2)
are satisfied for $p=2$.
\begin{corollary}\label{lem:sensitivity}
	Under the assumption of Proposition \ref{prop:At_mon_cont}
	there exist $t^* >0$ and $c>0$ such that 
	$$
		\|y^t - y\|_{H^1(\Omega)}\le ct \quad \text{ for all } t\in [0, t^*].
	$$
\end{corollary}

\subsection{Limiting system for the transformed material derivative}
In Corollary~\ref{lem:sensitivity} we have established a Lipschitz estimate for 
the mapping $t\mapsto y^t$.
In this section we are going to prove that there is a unique element
$\dot y$ in $H^1(\Omega)$ -- called the material derivative --
such that $(y^t-y)/t$ converges strongly 
to $\dot y$ in $H^1(\Omega)$ which is uniquely determined 
by a variational inequality. 

In order to derive the differentiability of $y^t$ we
impose the additional assumptions to (A1) and (A2):\vspace*{0.5em}\\
\textbf{Assumption (A3)}
\textit{
  \begin{itemize}
  	\item[(i)]
	  	$w(x,\cdot)$ is of class $C^1(\R)$
	  	for all $ x\in \Omega$;
  	\item[(ii)]
			for all $X\in C^1_c(D, \R^d)$, there exists a
			function $\dot w_{X}:\Omega\times\R\to\R$ such that
			for all $\ph_n\to\ph$ strongly in $H^1(\Omega)$
			we have
  	 	$\dot w_{X}(\cdot, \ph)\in L_{(2^*)'}(\Omega)$ and
  	 	for all $t_n\searrow 0$
			\begin{align*}
				\frac{w^{t_n}_{X}(\cdot, \ph_n) - w(\cdot, \ph_n)}{t_n} &\rightarrow
				\dot w_{X}(\cdot, \ph) && \text{ strongly in } L_{(2^*)'}(\Omega)\text{ as } n\to\infty;
			\end{align*}
    \item[(iii)]
    	for any given sequences
    	$\ph_n\to\ph$ in $H^1(\Omega)$ and
    	$t_n\searrow 0$
    	with $(\ph_n-\ph)/t_n\weaklim z$ weakly in $H^1(\Omega)$:
			\begin{align*}
				\frac{w(\cdot,\ph_{n})-w(\cdot,\ph)}{t_n} \rightarrow{}&\partial_y w(\cdot,\ph)z
					&& \text{ strongly in } L_{(2^*)'}(\Omega)\text{ as }n\to\infty;
			\end{align*}
		\item[(iv)]
			for all $X\in C^1_c(D,\R^d)$ there exists a function $\dot\psi_{X}\in H^1(\Omega)$
			such that
			\begin{align*}
				&\frac{\psi_X^{t}-\psi}{t}\to\dot\psi_X
					&& \text{ strongly in } H^1(\Omega)\text{ as }t\searrow 0.
			\end{align*}
  \end{itemize}
}
\begin{remark}
\label{remark:contA}
	\begin{itemize}
		\item[(i)]
			Property (iii) from Assumption (A3) is satisfied if, e.g., 
			there exist a constant $C>0$ and a function $s\in L_{\frac{2^*-1}{2^*-2}}(\Omega)$ such that for all $x\in \Omega$ and $y\in \R$:
			\begin{align*}
				|\partial_y w(x,y)|\leq C|y|^\alpha+s(x)
			\end{align*}
			with the exponent $\alpha:=\frac{2^*(2^*-1)}{2^*-2}$.
			The constant $\alpha$ is chosen such that the function
			$x\mapsto\partial_y w(x,\ph(x))z(x)$
			and $x\mapsto f'(\ph(x))z(x)$
			are in $L_{(2^*)'}(\Omega)$ for given $\ph,z\in H^1(\Omega)$.
		\item[(ii)]
			A useful consequence of properties (ii) and (iii) is the following
			continuity
			\begin{align*}
				&w_{X}^{t_n}(\cdot,\varphi_n)\to w(\cdot,\varphi)\quad
					\text{ strongly in } L_{(2^*)'}(\Omega)\text{ as }n\to\infty.
			\end{align*}
			for all $\ph_n\to\ph$ strongly in $H^1(\Omega)$
			and $t_n\searrow 0$.
		\item[(iii)]
			Let $X\in C^1_c(D, \R^d)$ be given. Then we have 
			by using property (iv) from Assumption (A3)
			\begin{align*}
			   \frac{-A(t)\nabla\psi_X^t+\nabla\psi}{t} \rightarrow{}&  -A'(0) \nabla \psi
			   - \nabla \dot\psi_X && \text{ strongly in } L_2(\Omega, \R^d),\\
			   \frac{-\xi(t)\psi_X^t+\psi}{t} \rightarrow{}&-\xi'(0)\psi - \dot\psi_X&& \text{ strongly in } L_2(\Omega, \R^d).
			\end{align*}
	\end{itemize}
\end{remark}
We are now well-prepared for the derivation of the material derivative.
\begin{theorem}
\label{thm:weakMatDer}
			Let (A1)-(A3) be satisfied.
      The weak material derivative $\dot y$  of $t \mapsto  y^t$ exists in 
      all directions $X\in C^1_c(D, \R^d)$ and is characterised as 
      the unique solution of the following variational inequality
      \begin{align}\label{eq:limiting}
       	\left\{
       	\begin{aligned}
       		&\dot y\in \tilde S_y(\Kc)\text{ and }\forall \ph\in\tilde S_y(\Kc):\\
       		&\langle \partial \tA(y)\dot y, \varphi - \dot y \rangle_{H^1} \ge - \langle \tA'(y), \varphi - \dot y\rangle_{H^1},
      	\end{aligned}
      	\right.
      \end{align}
			where $\tilde S_y(\Kc)$ denotes the closed and convex cone
      \begin{align}
				\tilde S_y(\Kc) = T_y(\Kc) \cap \text{kern}(\tA(y)).
				\label{SKdef}
      \end{align}
      The functional derivatives $\partial \tA$ and
      $\tA'$ are given by
      \begin{align} 
      	&\langle\partial\tA(y)\dot y,\ph\rangle
      		=\af( \dot y+\dot\psi_X, \ph)+\int_\Omega \partial_y w(x,y+\psi)\dot y\ph\dx,\\
      	&\begin{aligned}
	      	\langle\tA'(y),\ph\rangle={}&
	      		\int_\Omega A'(0)\nabla y\cdot\nabla\ph
	      		+\xi'(0)\big(\lambda y+w(x,y+\psi)\big)\ph\dx\\
	      		&+\int_\Omega\dot w_{X}(x,y+\psi)\ph+\partial_y w(x,y+\psi)\dot\psi_X\ph\dx\\
	      		&+\int_\Omega A'(0)\nabla\psi\cdot\nabla\ph+\xi'(0)\lambda\psi\ph\dx.
      	\end{aligned}
      \end{align}
\end{theorem}
\begin{proof}
\hspace*{0.01em}\\\textit{Existence of $\dot y$:}
	We want to apply Theorem~\ref{thm:mat_abstract}. For this we need to check 
	Assumption (O3).
	To this end we notice that by Corollary \ref{lem:sensitivity} $y_{t_n}\to u$ strongly and $(y_{t_n}-y)/t_n\weaklim z$ weakly in $H^1(\Omega)$
	for a suitable subsequence $t_n\searrow 0$.
	
	\begin{itemize}
		\item[$\bullet$]
		We check (O3) (ii):
		Let $v_n\weaklim v$ be a given weakly \red convergent \black sequence in $H^1(\Omega)$. Then
		\begin{align*}
			&\left \langle \frac{\tA_{t_n}(y^{t_n}) - \tA(y^{t_n})}{t_n}, v_n \right\rangle\\
			&=\underbrace{\int_\Omega\frac{A(t_n)-I}{t_n}\nabla y^{t_n}\cdot\nabla v_n\dx}_{\to \int_\Omega A'(0)\nabla y\cdot \nabla v\dx}
					+\underbrace{\int_\Omega\frac{\xi(t_n)-1}{t_n}\Big(\lambda y^{t_n}+w_{X}^{t_n}(x,y^{t_n}+\psi_X^{t_n})\Big)v_n\dx}_{
						\to \int_\Omega\xi'(0)(\lambda y+w(x,y+\psi))v\dx\quad\text{by Remark \ref{remark:contA} (ii)}}\\
				&\quad
					+\underbrace{\int_\Omega\frac{w_{X}^{t_n}(x,y^{t_n}+\psi_X^{t_n})-w(x,y^{t_n}+\psi_X^{t_n})}{t_n}v_n\dx}_{
						\to \int_\Omega\dot w_{X}(x,y+\psi)v\dx\quad\text{by Assumption (A3) (ii) and (iv)}}\\
				&\quad+\underbrace{\int_\Omega\frac{w(x,y^{t_n}+\psi_X^{t_n})-w(x,y+\psi)}{t_n}v_n\dx
					-\int_\Omega\frac{w(x,y^{t_n}+\psi)-w(x,y+\psi)}{t_n}v_n\dx}_{
						\to \int_\Omega\partial_y w(x,y+\psi)(z+\dot\psi_X)v\dx-\int_\Omega\partial_y w(x,y+\psi)zv\dx\,=\,\int_\Omega\partial_y w(x,y+\psi)\dot\psi_X v\dx\;\text{by (A3) (iii)-(iv)}}\\
				&\quad+\underbrace{\int_\Omega\frac{A(t_n)-I}{t_n}\nabla\psi_X^{t_n}\cdot\nabla v_n
					+\frac{\xi(t_n)-1}{t_n}\psi_X^{t_n}v_n\dx}_{
						\to -\int_\Omega A'(0)\nabla\psi\cdot\nabla v+\xi'(0)\psi v\dx}.
		\end{align*}
		\item[$\bullet$]
		We check (O3) (iii):
		\begin{align*}
			&\left\langle \frac{\tA(y^{t_n}) - \tA(y)}{t_n}, \frac{y^{t_n}-y}{t_n} \right\rangle\\
			&=\underbrace{\int_\Omega\Big|\nabla \frac{y^{t_n}-y}{t_n}\Big|^2
				+\lambda\Big|\frac{y^{t_n}-y}{t_n}\Big|^2\dx}_{\liminf\,\geq\,\int_\Omega |\nabla z|^2+\lambda |z|^2\dx}
				+\underbrace{\int_\Omega\frac{w(x,y^{t_n}+\psi)-w(x,y+\psi)}{t_n}\frac{y^{t_n}-y}{t_n}\dx}_{\to\int_\Omega \partial_y w(x,y+\psi)|z|^2\quad\text{by Assumption (A3) (iii)}}
		\end{align*}
		and for all $\ph_n\to\ph$ strongly in $H^1(\Omega)$:
		\begin{align*}
			&\left\langle \frac{\tA(y^{t_n}) - \tA(y)}{t_n}, \ph_n \right\rangle\\
			&=\underbrace{\int_\Omega\nabla \frac{y^{t_n}-y}{t_n}\cdot\nabla\ph_n
				+\lambda\frac{y^{t_n}-y}{t_n}\ph_n\dx}_{\to\,\int_\Omega \nabla z\cdot\nabla\ph+\lambda z\ph\dx}
				+\underbrace{\int_\Omega\frac{w(x,y^{t_n}+\psi)-w(x,y+\psi)}{t_n}\ph_n\dx}_{
					\to\int_\Omega \partial_y w(x,y+\psi)z\ph\quad\text{by Assumption (A3) (iii)}}.
		\end{align*}
		\item[$\bullet$]
		Property (O3) (i) follows from the above calculations.
	\end{itemize}
\textit{Uniqueness of $\dot y$:}
Assume two solutions $\dot y$ and $\dot z$ for \eqref{eq:limiting}.
Testing their variational inequalities with $\dot z$ and $\dot y$, respectively, and adding the result yields
\begin{align*}
	\langle \partial\tA(y)\dot y-\partial\tA(y)\dot z, \dot y-\dot z\rangle\leq 0.
\end{align*}
The left-hand side calculates as
\begin{align*}
	&\langle \partial\tA(y)\dot y-\partial\tA(y)\dot z, \dot y-\dot z\rangle\\
  &\quad=\af(\dot y-\dot z, \dot y-\dot z)+\int_\Omega\partial_y w(x,y+\psi)|\dot y-\dot z|^2\dx.
\end{align*}
Due to the convexity assumption in (A1) (i) we find $\partial_y w\geq 0$
and see that
\begin{align*}
  \af(\dot y-\dot z, \dot y-\dot z)
  	\leq 0.
\end{align*}
We obtain $\dot y-\dot z=0$.
\end{proof}

By exploiting the specific structure of $\tilde{\C A}_t$ and Assumption (A3)
we can even show that the strong material derivative exists.
\begin{corollary}
\label{cor:strongMatDer}
We have for all $X\in C^1_c(D, \R^d)$ 
\begin{align}
	\frac{ y_X^t - y}{t} \rightarrow \dot{y}_X\qquad\text{strongly in } H^1(\Omega).
\end{align}
\end{corollary}
\begin{proof}
We test the variational inequality \eqref{eq:energy_compact} with $\ph=y^t$
and for $t=0$ with $\ph=y$. Adding both inequalities yields
$$
	\langle\tA_t(y^t)-\tA(y),y^t-y\rangle\leq 0.
$$
Dividing by $t^2$ and rearranging the terms we obtain
by setting $z^t:= (y^t-y)/t$
\begin{equation}\label{eq:eq_yt_y0_}
	\begin{split}
	&\af(z^t, z^t)\\
	&\le  -\int_\Omega   \frac{A(t)-I}{t}\nabla  y^t \cdot \nabla z^t\dx
		-\int_\Omega \lambda\frac{\xi(t)-1}{t} y^t z^t \dx\\
  	&\quad-\int_\Omega \left( \frac{\xi(t)-1}{t}w^t_{X}(x,y^t+\psi_X^t) + \frac{w^t_{X}(x,y^t+\psi_X^t)-w(x,y^t+\psi_X^t)}{t}\right)z^t\dx\\
		&\quad-\int_\Omega \frac{w(x,y^t+\psi_X^t)-w(x,y+\psi)}{t}z^t \dx\\
		&\quad- \int_\Omega \frac{A(t)\nabla\psi_X^t-\nabla\psi}{t}\cdot \nabla z^t\dx
		- \int_\Omega \lambda\frac{\xi(t)\psi_X^t-\psi}{t}z^t\dx\\
		&\quad=:\frak B(t).
\end{split}
\end{equation}
The known convergence properties shows as $t\searrow 0$ for a subsequence
\begin{align*}
	\frak B(t)\to{}&\underbrace{-\langle\tA'(y),\dot y\rangle
		-\int_\Omega \partial_y w(x,y+\psi)|\dot y|^2\dx
		- \int_\Omega \nabla\dot\psi_X\cdot\nabla \dot y\dx
		- \int_\Omega \lambda\dot\psi_X\dot y\dx}_{=:\frak B(0)}.
\end{align*}
However testing \eqref{eq:limiting} with $\ph=2\dot y\in \tilde S_y(K)$
we also obtain
$\langle \partial \tA(y)\dot y, \dot y \rangle_{H^1} \ge - \langle \tA'(y), \dot y\rangle_{H^1}$
which is precisely
\begin{equation*}
	\af(\dot y,\dot y)\geq \frak B(0).
\end{equation*}
All in all we get
\begin{align}
	\limsup_{t\searrow 0} \af(z^t, z^t)
	\leq \limsup_{t\searrow 0} \frak B(t)
	=\frak B(0)\leq \af(\dot y,\dot y).
	\label{limsupEst}
\end{align}
The weak convergence $z^t\weaklim \dot y$ in $H^1(\Omega)$
implies $\liminf_{t\searrow 0}\af(  z^t , z^t)\ge \af(\dot y,\dot y)$.
Together with \eqref{limsupEst} this gives
$\af(z^t,z^t)\to \af(\dot y,\dot y)$ as $t\searrow 0$.
This finishes the proof.
\end{proof}

\begin{remark}
\label{rmk:VIL}
	If we assume that
  \begin{align}
  	\dot w_{X}(x, y ) :=  \Tb_0(x, y )\cdot X(x)+\Tb_1(x, y ):\partial X(x)
  	\label{assumptionDotW}
  \end{align}
  for functions
	$\Tb_0(\cdot, \cdot):\Omega\times \R \rightarrow \R^{d}$
	and $\Tb_1(\cdot, \cdot):\Omega\times \R \rightarrow \R^{d\times d}$
	we may rewrite the
	variational inequality in \eqref{eq:limiting}
	by using  Lemma \ref{lem:convergence_comp}
	as
  \begin{align*}
    &\af( \dot y,  \ph - \dot y ) +
    		\int_\Omega \partial_y w(x,y+\psi)\dot y(\ph - \dot y)\dx \\
    	&\quad\ge \int_\Omega \mathfrak L_1(x,y+\psi;\ph - \dot y) : \partial X  
  	  	+ \mathfrak L_0(x,y+\psi;\ph - \dot y)\cdot X \dx\\
    	&\qquad - \af( \dot \psi,  \ph - \dot y ) + 
				\int_\Omega \partial_y w(x,y+\psi) \dot \psi(\ph - \dot y)\dx,
  \end{align*}
  where we use the abbreviations
  \begin{align*}
    \mathfrak L_1(x,y+\psi;\ph)  := & -  \left( \nabla (y+\psi) \cdot \nabla \ph
    	+ \Big(\lambda(y+\psi) +  w(x,y+\psi)\Big)\ph \right) I \\
    	& + \nabla \ph \otimes \nabla (y+\psi) + \nabla (\psi+y) \otimes \nabla \ph
    	-\Tb_1(x,y+\psi)\ph,\\
  	\mathfrak L_0(x,y+\psi;\ph)  := & -\Tb_0(x,y+\psi)\ph.
	\end{align*}
\end{remark}

\subsection{Limiting system for the material derivative}
So far we have derived an equation for $\dot y$.
Since we are interested in the original problem \eqref{eq:VI_weak},
we may now use Theorem \ref{thm:weakMatDer} and
the transformation $y= u - \psi$ to obtain the material 
derivative equation for \eqref{eq:VI_weak}. It is 
clear that $\dot y = \dot u - \dot\psi_X$ and we conclude with the following result:
\begin{corollary}
\label{cor:dotU}
	Under the assumptions (A1)-(A3)
  the material deriative $\dot u=\dot u(X)$ of solutions of the perturbed problem to
 	\eqref{eq:VI_weak} in direction $X\in C^1_c(D, \R^d)$ exists and is 
	given as the solution of the following variational inequality:
  \begin{align}\label{eq:limiting2}
  	\left\{
   	\begin{aligned}
   		&\dot u\in S_u^X(K_\psi)\text{ and }\forall \ph\in S_u^X(K_\psi):\\
   		&\af( \dot u,  \ph - \dot u )
	   		+ \int_\Omega  \partial_y w(x,u)\dot u(\ph - \dot u)\dx\\
	   	&\quad
	   		\geq -\int_\Omega A'(0)\nabla u\cdot\nabla(\ph - \dot u)
	   		+ \xi'(0)\big(\lambda u+w(x,u)\big)(\ph - \dot u)\dx\\
	   	&\qquad
	   		-\int_\Omega\dot w_{X}(x,u)(\ph-\dot u)\dx,
   	\end{aligned}
   	\right.
  \end{align}
 	where 
  $$
  	S_u^X(K_\psi) := T_{u}(K_\psi)\cap kern(\C A(u))+\dot\psi_X.
  $$
	In particular under the additional assumption in Remark \ref{rmk:VIL}
  \begin{align*}
   		&\af( \dot u,  \ph - \dot u )
	   		+ \int_\Omega  \partial_y w(x,u)\dot u(\ph - \dot u)\dx \\
 			&\qquad\ge \int_\Omega \mathfrak L_1(x,u;\ph - \dot u) : \partial X  
      	+ \mathfrak L_0(x,u;\ph - \dot u)\cdot X \dx.
  \end{align*}
\end{corollary}
\begin{proof}
	We obtain from Theorem \ref{thm:weakMatDer} that
	$\dot u\in \tilde S_y(\Kc)+\dot\psi_X$ and for all
	$\ph\in\tilde S_y(\Kc)+\dot\psi_X$:
	$$
	  \langle \partial \tA(u-\psi)(\dot u-\dot\psi_X), \varphi - \dot u \rangle_{H^1} \ge - \langle \tA'(u-\psi), \varphi - \dot u\rangle_{H^1},
	$$
	which is precisely the inequality in \eqref{eq:limiting2}.
	
	It remains to show $S_u^X(K_\psi)=\tilde S_y(\Kc)+\dot\psi_X$
	which is equivalent to $T_u(K_\psi)\cap kern(\C A(u))=T_y(K)\cap kern(\tA(y))$.
	Indeed, by definition \eqref{tildeAdef} we find
	$$
		kern(\C A(u))=kern(\tA(y))
	$$
	as well as by \eqref{clarke_cone}-\eqref{normal_cone}
	\begin{align*}
		T_u(K_\psi)=T_{u-\psi}(K)
			= T_y(K)
	\end{align*}
\end{proof}
Note that we get the following characterisation of $S_u^X$ by using Theorem\ref{thm:TN} and the definition in \eqref{SKdef}:
\begin{align*}
	\ph\in S_u^X(K_\psi)
	&\quad\Leftrightarrow\quad
		\ph-\dot\psi_X\in T_{u}(K_\psi)\cap kern(\C A(u))\\
	&\quad\Leftrightarrow\quad
		\begin{cases}
			&\ph\in H^1(\Omega)\text{ with }\ph\leq\dot\psi_X\text{ q.e. on }\{u=\psi_\Omega\},\\
			&\langle\C A(u),\ph-\dot\psi_X\rangle=0.
		\end{cases}
\end{align*}
Moreover under an additional assumptions we obtain the subsequent translation property:
\begin{lemma}
\label{lemma:dotPsiInc}
	Suppose that $u,\psi\in H^2(\Omega)$ and
	let $\zeta\in H^1(\Omega)$ be with
	\begin{align*}
		\tilde\zeta=0\text{ q.e. on the coincidence set }\{x\in\ol\Omega:\,\tilde u(x)=\tilde \psi(x)\},
	\end{align*}
	where $\tilde\zeta$, $\tilde u$ and $\tilde\psi$ denote quasi-continuous representatives for $\zeta$, $u$ and $\psi$.
	Then we have
	$$
		\pm\zeta\in T_{u}(K_\psi)\cap kern(\C A(u)).
	$$
	In particular
	\begin{align}
		\zeta+S_u^X(K_\psi)=\C S_u^X(K_\psi).
		\label{eqn:pmInc}
	\end{align}
\end{lemma}
\begin{proof}
	It is clear from the assumption that $\pm\tilde\zeta=0$
	q.e. on the coincidence set $\{u=\psi\}$.
	Thus $\pm\zeta\in T_{u}(K_\psi)$.
	Furthermore $y=u-\psi$ satisfies the variational inequality (see \eqref{eq:energy_compact} with $t=0$)
	$$
		\langle\tA(y),\ph-y\rangle\geq 0\quad
		\text{for all $\ph\in H^1(\Omega)$ and $\ph\leq 0$ a.e. in $\Omega$.}
	$$
	From the $H^2(\Omega)$-regularity of $u$ and $\psi$
	we deduce
	that (in a pointwise formulation)
	$\tA(y)=0$ a.e. in $\{x\in\Omega:\,u(x)<\psi(x)\}$.
	In particular we see that
	$$
		\langle\tA(y),\ph\rangle=0
		\text{ for all }\ph\in H^1(\Omega)\text{ with }\{x\in\Omega:\,\ph(x)=0\}\supseteq \{x\in\Omega:\,u(x)=\psi(x)\}\text{ a.e.}
	$$
	Testing with $\ph=\pm\zeta$ yields $\pm\zeta\in kern(\tA(y))=kern(\C A(u))$.
	
	Finally, $\zeta\in T_{u}(K_\psi)\cap kern(\C A(u))$ implies $\zeta+S_u^X(K_\psi)\subseteq S_u^X(K_\psi)$,
	and $-\zeta\in T_{u}(K_\psi)\cap kern(\C A(u))$ implies $\zeta+S_u^X(K_\psi)\supseteq S_u^X(K_\psi)$.
\end{proof}
In the following $\psi_\Omega$ is referred to as a static obstacle
if there exists a fixed function $\psi\in H^2(D)$ such that $\psi_{\tilde\Omega}=\psi|_{\tilde\Omega}$
for all Lipschitz domains $\tilde\Omega\subseteq D$.
\begin{remark}
		Let $X\in C^1_c(D,\R^d)$.
		Suppose that $\psi_\Omega$ is a static obstacle, $u\in H^2(\Omega)$
		and $\{X=\mathbf 0\}\supseteq \{\tilde u=\tilde \psi_\Omega\}$ q.e. in $\ol\Omega$.
		Then $\dot\psi_X=\nabla\psi_\Omega\cdot X$ and the assumptions from Lemma \ref{lemma:dotPsiInc} are satisfied
		for $\zeta=\dot\psi_X$ and we obtain
		\begin{align*}
			\pm\dot\psi_X\in T_{u}(K_\psi)\cap kern(\C A(u)).
		\end{align*}
		In particular
		$$
			S_u^X(K_\psi) = T_{u}(K_\psi)\cap kern(\C A(u))
		$$
		and
		\begin{align*}
			\ph\in S_u^X(K_\psi)
			\quad\Leftrightarrow\quad
				\begin{cases}
					&\ph\in H^1(\Omega)\text{ with }\ph\leq 0\text{ q.e. on }\{u=\psi_\Omega\},\\
					&\langle\C A(u),\ph\rangle=0.
				\end{cases}
		\end{align*}
\end{remark}

%
%

\subsection{Limiting system for the state-shape derivative}
The \emph{state shape derivative} of $u$ at $\Omega$ in direction $X\in C^1_c(D, \R^d)$ is defined by 
\begin{align}
	u' =u'(X): = \dot u - \partial_X u \quad \text{ on } \Omega
	\label{eqn:shapeDerivative}
\end{align}
where $u$ solves \eqref{eq:VI_weak}, $\dot u$ solves \eqref{eq:limiting2} and $\partial_X u := \nabla u\cdot X$. It 
is clear that $u'\in L_2(\Omega)$.
Thus in general the state shape derivative is 
less regular than the material derivative. Another important observation is that
the boundary conditions imposed on $\dot u$ on $\partial \Omega$ are not carried over to $ u'$. 

\begin{lemma}\label{lem:state_shape}
Let $X\in C^1_c(D,\R^d)$ be a vector field satisfying $X\cdot n=0$ on 
$\partial \Omega$.  Then the state shape derivative vanishes identically, 
that is, $u'(X)=0$ a.e. on $\Omega$. 
\end{lemma}
\begin{proof}
        The $X$-flow $\Phi_t$ leaves the domain $\Omega$ unchanged, i.e., 
        $\Phi_t(\Omega) = \Omega$ for all $t\in[0,\tau]$.  Consequently, $u_t = u(\Omega_t) = u(\Omega) = u$ and thus
        $ u^t = u_t\circ \Phi_t=u\circ\Phi_t$ for all $t\in[0,\tau]$.
        Hence by Lemma \ref{lemma:phit} (ii) we may calculate the material
        derivative $\dot u$ as
	\begin{align*}
		\frac{u^t-u}{t}=
		\frac{u\circ\Phi_t-u}{t}
		\to
		\partial_X u\quad\text{strongly in }L_2(\Omega).
	\end{align*}
	Thus $\dot u=\partial_X u$ and consequently $u'=0$.
\end{proof}

Now we are prepared to prove the main result of this section
which gives a simplified variational inequality for the
state-shape derivative $u'$ under certain conditions.
To derive this result we will assume the enhanced regularity $u\in H^2(\Omega)$.
Preliminarily we observe
from Corollary \ref{cor:dotU} and by using the relation \eqref{eqn:shapeDerivative} that
\begin{align}
  	\left\{
   	\begin{aligned}
   		&u'\in \hat S_u^X(K_\psi)\text{ and }\forall \ph\in\hat S_u^X(K_\psi):\\
   		&\af( u',  \ph - u')
	   		+ \int_\Omega  \partial_y w(x,u) u'(\ph - u')\dx\\
	   	&\quad
	   		\geq -\int_\Omega A'(0)\nabla u\cdot\nabla(\ph - u')
	   		+ \xi'(0)\big(\lambda u+w(x,u)\big)(\ph - u')\dx\\
	   	&\quad\quad
		   	-\af( \partial_X u,  \ph - u')
		   	-\int_\Omega  \partial_y w(x,u) \partial_X u(\ph - u')\dx,
   	\end{aligned}
   	\right.
   	\label{eqn:uDerVI}
\end{align}
where 
$$
	\hat S_u^X(K_\psi) := T_{u}(K_\psi)\cap kern(\C A(u)) + \dot\psi_X-\partial_X u.
$$
We notice that in general the cone $\widehat{\mathcal{S}}_X( K)$
depend on the vector field $X$.
In the case of a static obstacle problem (see Remark \ref{rmk:setting} (i))
we derive the following result:
\begin{theorem}
\label{thm:shapeDer}
	Suppose that (A1)-(A3), \eqref{assumptionDotW} and $u\in H^2(\Omega)$ hold.
	Furthermore let $\psi_\Omega$ be a static obstacle function.
	
	\red Then $\pm(\dot\psi_X-\partial_X u)\in T_{u}(K_\psi)\cap kern(\C A(u))$ and
	we have\black
	\begin{align}
		\hat S_u^X(K_\psi)=T_{u}(K_\psi)\cap kern(\C A(u))
		\label{eqn:Sidentity}
	\end{align}
	\red In particular $\hat S_u^X(K_\psi)$ is independent of $X$
	and we write
	$
		S_u(K_\psi)=\hat S_u^X(K_\psi).
	$
	Furthermore \black the
	state shape derivative
	is the unique solution of
	\begin{align}\label{eq:state-shape-derivative}
	\left\{
	\begin{aligned}
 		&u'\in S_u(K_\psi)\text{ and }\forall \ph\in S_u(K_\psi):\\
		&\af( u',  \ph - u' ) + \int_{\Omega} \partial_y w(x,u)  u'(\ph -  u')\dx\\
		&\quad\ge \int_{\Gamma} \mathfrak S_1(x,u;\ph - u')n\cdot n (X\cdot n)\ds,
	\end{aligned}
	\right.
	\end{align}
	where 
	\begin{align*}
		\mathfrak S_1(x,u;\ph) :=  \mathbf{\mathfrak{L}}_1(x,u;\ph)  - \nabla u\otimes \nabla \ph.
	\end{align*}
	with $\mathbf{\mathfrak{L}}_1$ from Remark \ref{rmk:VIL}.
\end{theorem}
\begin{proof}
	By using the assumption $\dot\psi_X=\nabla\psi_\Omega\cdot X$ we find
	on the coincidence set $\{u=\psi_\Omega\}$
	(here we resort to quasi-continuous representants):
	$$
		\dot\psi_X-\partial_X u=\dot\psi_X-\partial_X \psi_\Omega=0.
	$$
	Lemma \ref{lemma:dotPsiInc} applied to $\zeta=\dot\psi_X-\partial_X u$ yields
	$\pm(\dot\psi_X-\partial_X u)\in S_u(K_\psi)$ and therefore
	\eqref{eqn:Sidentity}.
	
	Furthermore by using the notation in Remark \ref{rmk:VIL}
	and the identity (note that $u\in H^2(\Omega)$ by assumption)
	$$
		\nabla(\partial_X u)=(\partial X)^\top(\nabla u)+(\partial^2 u)X,
	$$
	the variational inequality in \eqref{eqn:uDerVI}
	rewrites to
	$u'\in S_u(K_\psi)$ and for all $\ph\in S_u(K_\psi)$:
	\begin{equation}\label{eq:limiting_y_dot}
		\begin{split}
			\af( u',  \ph - u' ) + &  \int_\Omega  \partial_y w(x,u)  u'(\ph -  u')\dx \\
			 \ge & \int_\Omega \mathfrak S_1(x,u;\ph - u') : \partial X + 
			 \mathfrak S_0(x,u;\ph - u') \cdot X \dx,
	\end{split}
	\end{equation}
	where
	\begin{align*}
	\mathfrak S_0(x,u,\ph) & := \mathfrak L_0(x,u,\ph)  - \partial_y w(x,u) \ph\nabla u
		-(\partial^2 u)\nabla\ph,\\
	\mathfrak S_1(x,u,\ph) & := \mathfrak L_1(x,u,\ph) - \nabla u\otimes \nabla 
	\ph. 
	\end{align*}
	Picking any vector field $X\in C^1_c(D, \R^d)$ with $X\cdot n=0$ on $\Gamma$
	we know from 
	Lemma~\ref{lem:state_shape} that $u'(\pm X) =0$ and
	it follows from \eqref{eq:limiting_y_dot}
	 \begin{equation}\label{eq:F} 
		\begin{split}
	  \int_\Omega \mathfrak S_1(x,u;\tilde\ph) : \partial X + \mathfrak S_0(x,u;\tilde\ph)
	   \cdot X \dx =0 
	\end{split}
	\end{equation} 
	for all $\tilde\ph\in S_u(K_\psi)$. 
	Then integrating by parts in \eqref{eq:F} shows the pointwise identity
	\begin{align}
		-\divv(\mathfrak S_1(x,u(x);\tilde\ph(x))) + \mathfrak S_0(x,u(x);\tilde\ph(x)) =0 \quad \text{ a.e. on } \Omega.
		\label{eqn:Spointwise}
	\end{align}
	Now for an arbitrary $X\in C_c^1(D, \R^d)$ and $\tilde\ph\in S_u(K_\psi)$
	we consider the additive splitting $X=X_n+X_T$
	for $X_n,X_T\in C_c^1(D, \R^d)$
	such that $X_n=n(X\cdot n)$
	and $X_T=X-n(X\cdot n)$ on $\Gamma$.
	Then $X_T\cdot n=0$ on $\Gamma$ and we get
	\begin{align}
		\begin{aligned}
			&\int_\Omega \mathfrak S_1(x,u;\tilde\ph) : \partial X + \mathfrak S_0(x,u;\tilde\ph)\cdot X\dx\\
			&=\underbrace{\int_\Omega \mathfrak S_1(x,u;\tilde\ph) : \partial X_T + \mathfrak S_0(x,u;\tilde\ph)\cdot X_T\dx}_{=0\text{ by \eqref{eq:F}}}\\
			&\quad\underbrace{+\int_\Omega \mathfrak S_1(x,u;\tilde\ph) : \partial X_n + \mathfrak S_0(x,u;\tilde\ph)\cdot X_n\dx}_{\text{partial integration and \eqref{eqn:Spointwise}}}\\
			&= \int_{\Gamma} \mathfrak S_1(x,u;\tilde\ph)n \cdot X_n\ds.
		\end{aligned}
		\label{eqn:Sformula}
	\end{align}
	We may test \eqref{eqn:Sformula} with $\tilde\ph=u'$
	since $u'\in S_u(K_\psi)$. Then multiplying the resulting identity with $-1$
	and exploiting linearity of $\mathfrak S_0$ and $\mathfrak S_1$ with respect to $\ph$
	yields
	\begin{align}
		\int_\Omega \mathfrak S_1(x,u;-u') : \partial X + \mathfrak S_0(x,u;-u')\cdot X\dx  = 
			\int_{\Gamma} \mathfrak S_1(x,u;-u')n \cdot X_n\ds.
		\label{eqn:Sformula2}
	\end{align}
	Now we find by letting $\tilde\ph=\ph\in S_u(K_\psi)$ be arbitrary, adding
	\eqref{eqn:Sformula} and \eqref{eqn:Sformula2}, and again exploiting linearity
	\begin{align*}
		\int_\Omega \mathfrak S_1(x,u;\ph-u') : \partial X + \mathfrak S_0(x,u;\ph-u')\cdot X\dx  = 
			\int_{\Gamma} \mathfrak S_1(x,u;\ph-u')n \cdot X_n\ds.
	\end{align*}
	In combination with \eqref{eq:limiting_y_dot} we obtain
	\eqref{eq:state-shape-derivative}.
	Uniqueness of $u'$ is implied by uniqueness of $\dot y$ (see Theorem \ref{thm:weakMatDer}).
\end{proof}

It is readily checked that 
$$
	\mathfrak S_1(x,u;\ph)n\cdot n = - \nabla_\Gamma u \cdot \nabla_\Gamma \ph - \big(\lambda u+w(x,u)\big)\ph  \quad \text{ for all } \ph \in S_u(K_\psi).
	$$
Thus we conclude this section with an explicit formula
for the shape derivative in the case of a static obstacle.
\begin{corollary}
\label{cor:stateShapeDer}
Under the assumption of Theorem \ref{thm:shapeDer}
the shape derivative $u'$ is the unique solution of the following 
variational inequality:
\begin{align*}
		u' \in S_u(K_\psi), \quad 	&	\af( u',  \ph - u' )
		+ \int_{\Omega} \partial_y w(x,u)  u'(\ph -  u')\dx \\
		&\quad\ge -\int_{\Gamma} \big[\nabla_\Gamma 
			u \cdot \nabla_\Gamma (\ph -  u') +  \big(\lambda u+w(x,u) \big)\big](X\cdot n)(\ph- u')\ds
\end{align*}
for all $\ph \in S_u(K_\psi)$.
\end{corollary}

\subsection{Eulerian semi-derivative of certain shape functions}
We adopt the notation from Appendix B
and denote by $J:\Xi\to\R$ a shape function.
Application of Corollary \ref{cor:dotU}, Lemma \ref{lemma:pos1hom}
and the chain rule yield the following result:
\begin{corollary}
\label{cor:EulerSemi}
	Let (A1)-(A3) be satisfied and let $\Omega\in\Xi$ be a Lipschitz domain,
	$X\in C^1_c(D,\R^d)$ and $\Phi_t:\Omega\to\Omega_t$ be the associated flow.
	Suppose that for all small $t>0$
	$$
		J(\Omega_t)=\frak J(\Phi_t,u^t),
	$$
  where
  $$
  	\frak J=\frak J(\Phi,u):C^{0,1}(\Omega;\R^d)\times H^1(\Omega)\to \R
  $$
  is assumed to be a Fr\'echet differentiable functional
	and $u^t\in H^1(\Omega)$ the transported state $u^t=u_t\circ\Phi_t$
	with the unique solution $u_t$ of \eqref{eq:VI_weak} on $\Omega_t$.
	
	Then the Eulerian semi-derivative exists and is given as
	\begin{align*}
		dJ(\Omega)(X)=\langle d_\Phi \frak J(Id,u^0),X\rangle_{C^{0,1}(\Omega;\R^d)}
			+\langle d_u \frak J(Id,u^0),\dot u_X\rangle_{H^1(\Omega)},
	\end{align*}
	where
	$\dot u_X$ denotes the unique solution of \eqref{eq:limiting2}.
	
	\red In particular $dJ(\Omega)(\cdot)$ is positively 1-homogeneous, which could be further exploited for numerical purposes. \black
\end{corollary}

\section{Applications to damage phase field models}
\label{sec:damage}
	In this section we investigate
	shape optimisation problems for	
	a coupled inclusion/pde system describing
	damage processes in linear elastic materials.
	Our aim is to apply the abstract results from Section \ref{sec:semilinProb}
	designed for semilinear variational inequalities
	with dynamic obstacles to such concrete application scenarios.
	In this way we demonstrate how necessary optimality conditions
	for shape problems can be derived for relevant engineering tasks.

\subsection{Physical model}
	The physical model under consideration 
	was derived in \cite{FN96} and is
	described in the time-continuous setting by the following relations:
	\begin{subequations}
	\label{contSystem}
	\begin{align}
		&\ub_{tt}-\DIV\big(\mathbb C(\chi)\e(\ub)\big)=\bell,\\
		&0\in \partial I_{(-\infty,0]}(\chi_t)+\chi_t-\Delta\chi+\frac12\mathbb C'(\chi)\e(\ub):\e(\ub)+\red g'(\chi),
		\label{contSystem2}
	\end{align}
	\end{subequations}
	with the damage-dependent stiffness tensor $\mathbb C$ and the damage potential function $f$.
	The variable $\ub$ denotes the displacement field,
	$\e(\ub):=\frac{1}{2}(\partial\ub+(\partial \ub)^\top)$ the linearised strain tensor
	and $\chi$ is an internal variable
	(a so-called phase field variable) indicating the degree of damage.
	In terms of damage mechanics $\chi$ is interpreted as the density
	of micro-defects and is therefore valued in the unit interval (cf. \cite{LD05}).
	In this spirit we may use the following interpretation:
	\begin{align*}
		&\chi(x)=\begin{cases}
			1&\leftrightarrow\;\text{no damage in }x,\\
			\in(0,1)&\leftrightarrow\;\text{partial damage in }x,\\
			0&\leftrightarrow\;\text{maximal damage in }x.
		\end{cases}
	\end{align*}
	The system is supplemented with initial-time values for $\chi$, $\ub$ and $\ub_t$,
	Dirichlet boundary condition for $\ub$
	and homogeneous Neumann boundary condition for $\chi$.
	The governing state system \eqref{contSystem} can be derived by balance equations and suitable
	constitutive relations such that the laws of thermodynamics
	from continuum physics are fulfilled. We refer to \cite{FN96} for more details on the derivation of the model.
	
	A main feature of the evolution system \eqref{contSystem} is the
	uni-directionality constraint $\chi_t\leq 0$
	enforced by the subdifferential $\partial I_{(-\infty,0]}(\chi_t)$.
	This leads to non-smooth/switching behaviour of the evolution law
	by noticing that \eqref{contSystem2} rewrites as
	\begin{align*}
		\chi_t=
		\begin{cases}
			d,&\text{if }d\leq 0,\\
			0,&\text{if }d >0
		\end{cases}
		\quad\text{with the driving force}\quad
		d=\Delta\chi-\frac12\mathbb C'(\chi)\e(\ub):\e(\ub)-g'(\chi).
	\end{align*}
	A weak formulation of \eqref{contSystem} and existence of weak solution
	can be found in \cite{HK13} with minor adaption.
	Existence and uniqueness results for strong solutions
	for the above system with higher-order viscous terms
	are established in \cite{FH15}.
	For the analysis of quasi-linear variants of \eqref{contSystem} and for rate-independent
	as well as rate-dependent cases, we refer to \cite{KRZ15}
	and the references therein.

	The following remark justifies that the phase field variable $\chi$
	takes only admissible values provided $H^1(0,T;H^1(\Omega))$-regularity and mild growth assumptions on $\mathbb C$ and $g$.
	In that case
	it is not necessary to include
	a second sub-differential
	of the type $\partial I_{[0,1]}(\chi)$
	in \eqref{contSystem2} in order to force $\chi$ to be bounded in the unit interval.
	The precise assumptions for $\mathbb C$ and $g$ will be stated in (D1) below. At this point
	they are assumed to be continuously differentiable.
	\begin{remark}[Maximum principle]
	\label{remark:maxPrin}
		Suppose that $\mathbb C'(x)=\bm0$ and $g'(x)=0$ for all $x<0$.
		Then a weak solution $\chi\in H^1(0,T;H^1(\Omega))$ of \eqref{contSystem2} is always bounded in
		the unit interval as long as the initial-time value $\chi(0)=\chi^0$ is.
	\end{remark}
	\begin{proof}[Proof of Remark \ref{remark:maxPrin}]
		Because of $\chi_t(t)\leq 0$ for all times $t\in[0,T]$ and $\chi(0)\in [0,1]$ we
		obtain $\chi(t)\leq 1$.
		It remains to show $\chi(t)\geq 0$.
		
		Please notice that we cannot directly test \eqref{contSystem2}
		with $(\chi^-)_t$ since $\chi^-:=\min\{0,\chi\}$ is not necessarily in $H^1(0,T;H^1(\Omega))$
		even for smooth $\chi$.
		Instead, we test the inclusion \eqref{contSystem2} with $(m_\epsilon(\chi))_t$
		where $m_\epsilon$ denotes the following concave $C^{1,1}$-approximation of $\min\{0,\cdot\}$
		\begin{align*}
			m_\epsilon(x)=
				\begin{cases}
					x,&\text{if }x\in(-\infty,-\epsilon],\\
					-\frac{1}{16\epsilon}(x-3\epsilon)^2,&\text{if }x\in(-\epsilon,3\epsilon],\\
					0,&\text{if }x\in(3\epsilon,+\infty),
				\end{cases}
		\end{align*}				
		we obtain by simple rewriting
		\begin{align*}
			\begin{aligned}
				&\iint|\chi_t|^2 m_\epsilon'(\chi)+\nabla m_\epsilon(\chi)\cdot\nabla(m_\epsilon(\chi))_t
				+\big(\nabla\chi-\nabla m_\epsilon(\chi)\big)\cdot\nabla(m_\epsilon(\chi))_t\dxt\\
				&\quad+\iint\Big(\frac12\mathbb C'(\chi)\e(\ub):\e(\ub)+g'(\chi)+\xi\Big)m_\epsilon'(\chi)\chi_t\dxt=0,
			\end{aligned}
		\end{align*}
		where the function $\xi$ satisfies $\xi\in \partial I_{(-\infty,0]}(\chi_t)$ pointwise.
		We obtain by noticing that $m_\epsilon(\cdot)\to(\cdot)^-:=\min\{\cdot,0\}$ strongly in $H^{1}(\R)$
		and weakly-star in $W_\infty^{1}(\R)$ as $\epsilon\searrow 0$:
		\begin{align*}
			\begin{aligned}
				&\underbrace{\iint|\chi_t|^2 m_\epsilon'(\chi)\dxt}_{\to \iint|(\chi^-)_t|^2\dxt}+
					\underbrace{\frac12\int_\Omega|\nabla m_\epsilon(\chi(t))|^2-|\nabla m_\epsilon(\chi(0))|^2\dx}_{
					\to\frac12\int_\Omega|\nabla \chi^-(t)|^2-|\nabla\chi^-(0)|^2\dx}\\
				&\quad+\underbrace{\iint\big(\nabla\chi\cdot\nabla\chi_t\big)(1-m_\epsilon'(\chi))m_\epsilon'(\chi)\dxt}_{
					\to0}
					+\iint\underbrace{|\nabla\chi|^2(1-m_\epsilon'(\chi)m_\epsilon''(\chi)\chi_t}_{\geq 0\text{ due to }m_\epsilon''\leq 0,\,\chi_t\leq 0,\,m_\epsilon'\in[0,1]}\dxt\\
				&\quad+
					\underbrace{\iint\Big(\frac12\mathbb C'(\chi)\e(\ub):\e(\ub)+g'(\chi)+\xi\Big)m_\epsilon'(\chi)\chi_t\dxt}_{
					\to \iint\big(\frac12\mathbb C'(\chi^-)\e(\ub):\e(\ub)+g'(\chi^-)+\xi\big)\chi^-_t\dxt}=0.
			\end{aligned}
		\end{align*}
		We have by assumption $\mathbb C'(\chi^-)=\bm0$ and $g'(\chi^-)=0$.
		Furthermore $\xi\times(\chi^-)_t=0$ since $\xi=0$ as long as $\chi_t <0$.
		All in all we find by passing to $\epsilon\searrow 0$
		\begin{align*}
			&\frac12\int_\Omega|\nabla \chi^-(t)|^2-|\nabla\chi^-(0)|^2\dx
				+\iint|(\chi^-)_t|^2\dxt\leq 0.
		\end{align*}
		Since $\chi^-(0)=0$ in $\Omega$ we find $\chi^-(t)=0$ in $\Omega$ for all times $t\in[0,T]$.
	\end{proof}
	In the next section we will consider a time-discrete version of \eqref{contSystem} where such a maximum principle can also be obtained.

\subsection{Setting up time-discretisation scheme and shape optimisation problem}
	The shape optimisation problems will be performed on a time-discrete version of \eqref{contSystem}
	and for two spatial dimensions.
	Let $\{0,\tau,2\tau,\ldots,\tau N\}$ be an equidistant partition
	of $[0,T]$.
	The positive parameter $\tau>0$ denotes the time step size.
	In the remaining part of this work we make use of the following assumptions:
	\vspace*{0.5em}\\
	\textbf{Assumption (D1)}
	\textit{
		\begin{itemize}
			\item[(i)]
				$d=2$;
			\item[(ii)]
				The damage-dependent stiffness tensor satisfies $\mathbb C(\cdot)=\mathsf{c}(\cdot)\CC$, where
				the coefficient function $\mathsf{c}$ is assumed to be
				of the form
				\begin{align*}
					\mathsf{c}=\mathsf{c}_1+\mathsf{c}_2\text{ where $\mathsf{c}_1\in C^{2}(\R)$ is convex and $\mathsf{c}_2\in C^{2}(\R)$ is concave.}
				\end{align*}
				Moreover, we assume that $\mathsf{c},\mathsf{c}_1',\mathsf{c}_1'',\mathsf{c}_2',\mathsf{c}_2''$ are bounded and
				as well as
				\begin{align*}
					\mathsf{c}(x)\geq \eta\qquad\text{ for all }x\in\R.
				\end{align*}
				with constant $\eta>0$.
				The 4$^\mathrm{th}$ order stiffness tensor $\CC\in\C L(\R_\mathrm{sym}^{n\times n};\R_\mathrm{sym}^{n\times n})$
				is assumed to be symmetric and positive definite, i.e.
				\begin{align*}
					\CC_{\red ijkl\black}=\CC_{\red jikl\black}=\CC_{\red klij\black}\text{ and }e:\CC e\geq \eta|e|^2\text{ for all }e\in \R_\mathrm{sym}^{n\times n};
				\end{align*}
			\item[(iii)]
				$g$ is assumed to be of the form
				\begin{align*}
					g=g_1+g_2\text{ where $g_1\in C^{2}(\R)$ is convex and $g_2\in C^{2}(\R)$ is concave.}
				\end{align*}
				 Moreover we assume $g_1'$ and $g_2'$ to be Lipschitz continuous;
			\item[(iv)]
					$\bell^k\in L_2(D;\R^2)$ for all $k=0,\ldots,N$;
			\item[(v)]
					$\db^k\in H^2(D;\R^2)$ for all $k=0,\ldots,N$;
			\item[(vi)]
					initial values: $\ub^0,\vb^0\in H^2(D;\R^2)$ and
					$\chi^0\in H^2(D)$.
		\end{itemize}
	}
	
	Let $\Omega\subseteq D$ be a given Lipschitz domain.
	In this section a time-discrete model to \eqref{contSystem} will be
	investigated in a \textit{thermodynamically consistent} scheme
	(in this context it indicates that the time-discrete energy-dissipation inequality is satisfied).
	A related time-discretisation scheme has been used in \cite{FH15}.
	For all $k\in\{1,\ldots,N\}$ we are looking for a weak solution of
	\begin{subequations}
	\label{discrSystem}
	\begin{align}
		&\frac{\ub^{k}-2\ub^{k-1}+\ub^{k-2}}{\tau^2}-\DIV\big(\mathbb C(\chi^k)\e(\ub^k)\big)=\bell^k,
			\label{discrEqU}\\
		&\begin{aligned}
				&0\in\partial I_{(-\infty,0]}\Big(\frac{\chi^k-\chi^{k-1}}{\tau}\Big)+\frac{\chi^k-\chi^{k-1}}{\tau}-\Delta\chi^k+g_1'(\chi^k)+g_2'(\chi^{k-1})\\
				&\qquad+\frac12\big(\mathsf c_1'(\chi^k)+\mathsf c_2'(\chi^{k-1})\big)\CC\e(\ub^{k-1}):\e(\ub^{k-1}).
			\end{aligned}
			\label{discrEqChi}
	\end{align}
	\end{subequations}
	In accordance with the time-continuous model from the previous section $\chi^0$, $\uu^{0}$ and $\uu^{-1}:=\uu^0+\tau \vb^{0}$
	are the initial values and the boundary conditions
	are chosen as
	\begin{align}
		\nabla\chi\cdot\nu =0,\quad\ub^k=\db^k\quad\text{ on }\partial\Omega.
	\end{align}
	For notational convenience we will write $\zb=\{\ub^k,\chi^k\}_{k=0}^N$.
	\begin{remark}
		\begin{itemize}
			\item[(i)] Existence of weak solutions for \eqref{discrSystem}
				can be obtained by alternate minimisation for each time step by firstly solving \eqref{discrEqChi}
				and then solving \eqref{discrEqU}.
				In particular the solution $\chi^k$ from \eqref{discrEqChi}
				is the unique minimiser of the strictly convex potential
				\begin{align*}
					F(\chi)={}&\io\frac12|\nabla\chi|^2+\frac\tau2\big|\frac{\chi-\chi^{k-1}}{\tau}\big|^2
					+\frac 12 \big(\mathsf{c}_1(\chi)+ \mathsf{c}_2'(\chi^{k-1})\chi\big)\CC\e(\ub^{k-1}):\e(\ub^{k-1})\dx\\
					&+\io g_1(\chi)+g_2'(\chi^{k-1})\chi\dx.
				\end{align*}
				over the convex set
				\begin{align*}
					 K^{k-1}:=\big\{\chi\in H^1(\Omega)\,:\,\chi\leq\chi^{k-1}\text{ a.e. in }\Omega\big\}.
				\end{align*}
				As we point out later a higher integrability result from \cite{FKR13}
				yields $\e(\ub)\in L_p(\Omega)$ for \red some \black$p>2$.
				In combination with the embedding $H^1(\Omega)\hookrightarrow L_q(\Omega)$
				for every $q\geq 1$ valid for $d=2$ and Assumption (D1) (ii), the potential term $\io\frac 12 \big(\mathsf{c}_1(\chi)+ \mathsf{c}_2'(\chi^{k-1})\chi\big)\CC\e(\ub^{k-1}):\e(\ub^{k-1})\dx$
				in $F$ is well-defined.
			\item[(ii)]
				Under the additional assumptions that $\mathsf c_1(x)\geq \mathsf c_1(0)$ and  $g_1(x)\geq g_1(0)$ for all $x\leq 0$
				as well as $\mathsf c_2'(x)\leq 0$ and $g_2'(x)\leq 0$ for all $x\in[0,1]$ we obtain that $F(\max\{\chi^k,0\})\leq F(\chi^k)$
				(cf. \cite[Proposition 4.1]{KRZ15}).
				Thus $\chi^k$ is bounded in the unit interval as long as $\chi^{k-1}$ is.
			\item[(iii)]
				The discretisation scheme above is motivated by the fact that
				the associated time-discrete energy-dissipation inequality
				is obtained by testing \eqref{discrEqU} with $\ub^k-\ub^{k-1}-(\db^k-\db^{k-1})$
				and \eqref{discrEqChi} with $\chi^k-\chi^{k-1}$, adding and using
				convexity and concavity estimates (cf. \cite[Lemma 2.9]{FH15}).
		\end{itemize}
	\end{remark}


	For the shape optimisation problem it is convenient to rewrite the
	pde/inclusion system \eqref{discrEqChi} as
	\begin{subequations}
	\begin{align}
		&\left.\begin{aligned}
			&\ub^k\in \db^k+\ac H^1(\Omega;\R^2), \text{ and }\forall\ph\in \ac H^1(\Omega;\R^2):\\
			&\int_\Omega\frac{\ub^{k}-2\ub^{k-1}+\ub^{k-2}}{\tau^2}\ph
				+\mathbb C(\chi^k)\e(\ub^k):\e(\ph)\dx
				=\int_\Omega\bell^k\cdot\ph\dx
		\end{aligned}
		\hspace*{6.63em}\right\}\hspace*{-0.5em}
		\label{discrWeakU}\\
		&\left.\begin{aligned}
			&\chi^k\in  K^{k-1}\text{ and }\forall\ph\in  K^{k-1}:\\
			&\int_\Omega\nabla\chi^k\cdot\nabla(\ph-\chi^k)+\frac{\chi^k-\chi^{k-1}}{\tau}(\ph-\chi^k)+\big(g_1'(\chi^k)+g_2'(\chi^{k-1})\big)(\ph-\chi^k)\dx\\
			&+\int_\Omega\frac12\big(\mathsf c_1'(\chi^k)+\mathsf c_2'(\chi^{k-1})\big)\CC\e(\ub^{k-1}):\e(\ub^{k-1})(\ph-\chi^k)\dx\geq 0.
		\end{aligned}
		\;\,\right\}\hspace*{-0.5em}
		\label{discrDamageVI}
	\end{align}
	\label{discrSystemWeak}
	\end{subequations}
	\red
	Here $\ac H^1(\Omega;\R^2)$ denotes the closure of $C^\infty_c(\Omega;\R^2)$ in the $\|\cdot\|_{H^1}$-norm.
	\black
	In other words the state system is given by \red $N$-\black coupled variational inequalities
	with dynamic obstacles for the $N$ time steps. The obstacles are determined as
	the solutions of
	the damage variational inequality from the previous time step.
	\vspace*{0.5em}\\
	\textbf{Statement of the shape optimisation problem}\\	
	Our aim is to determine an optimal shape $\Omega\in \Xi$
	from a suitable class of domains
	such that a tracking type cost functional
	\begin{align}
		J(\Omega,\zb(\Omega))=\frac{\lambda_\ub}{2}\sum_{k=1}^N\|\ub^k(\Omega)-\ub_r^k\|_{L_2(\Omega;\R^2)}^2
			+\frac{\lambda_\chi}{2}\sum_{k=1}^N\|\chi^k(\Omega)-\chi_r^k\|_{L_2(\Omega)}^2
			\label{JDef}
	\end{align}
	is minimised under the constraint that
	\begin{align}
		\zb(\Omega)\text{ solve \eqref{discrSystemWeak} on $\Omega$ for all }k\in\{1,\ldots,N\}.
			\label{shapePDE}
	\end{align}
	The functions $\ub_r^k\in L_2(D;\R^2)$ and $\chi_r^k\in L_2(D)$ for $k=1,\ldots,N$ are prescribed displacements and damage patterns.
	Since the state $\zb(\Omega)$ is uniquely determined by $\Omega$ we may equivalently say
	that we aim to minimise the shape function
	\begin{align}
		J(\Omega):=J(\Omega,\zb(\Omega)).
		\label{frakJDef}
	\end{align}
	\red
	Applications include minimisation of overall damage by choosing $\chi_r^k\equiv 1$
	as well as deliberately inducing damage at some desired areas which are encoded in $\chi_r^k$.
	\black


\subsection{Material derivative and necessary optimality system}
	Let us fix a vector field $X\in C_c^1(D,\R^2)$.
	In accordance with Section \ref{sec:semilinProb}
	the associated perturbed solutions of \eqref{discrEqU}-\eqref{discrEqChi}
	on $\Omega_t:=\Phi_t(\Omega)$ are denoted by $\zb_t=\{\ub_t^{k},\chi_t^{k}\}_{k=0}^N$
	whereas the transported perturbed solutions
	are indicated by $\zb^t=\{\ub^{k,t},\chi^{k,t}\}_{k=0}^N$.
	Note that $\zb^0=\zb$.
	
	We proceed inductively over $k=1,\ldots,N$
	and assume that the strong material derivatives
	at the time steps $k-1$ and $k-2$ exist, i.e. for a subsequence $t\searrow 0$
	\begin{subequations}
	\begin{align}
		&\frac{\ub^{k-1,t}-\ub^{k-1,0}}{t}\to \dot\ub^{k-1}\quad\text{strongly in }H^1(\Omega;\R^2),
			\label{discUConv}\\
		&\frac{\ub^{k-2,t}-\ub^{k-2,0}}{t}\to \dot\ub^{k-2}\quad\text{strongly in }H^1(\Omega;\R^2),
			\label{discUConv2}\\
		&\frac{\chi^{k-1,t}-\chi^{k-1,0}}{t}\to \dot\chi^{k-1}\quad\text{strongly in }H^1(\Omega).
			\label{discChiConv}
	\end{align}
	\end{subequations}

	\paragraph*{Material derivative for the $\bm\chi^{\bm k}$-variable}
	\hspace*{0.1em}\\
	We want to apply Corollary \ref{cor:dotU} which is based on Theorem \ref{thm:weakMatDer}
	to establish the material derivative for the $\chi^{k}$-variable
	and its variational inequality.

	To check that the Assumptions (A1)-(A3) are fulfilled we require
	higher integrability estimates for $\ub^{k-1,t}$.
	Note that $\ub^{k-1,t}$ satisfies equation \eqref{pertDiscUEq} below for $k-1$
	which is the unique minimiser of
	\begin{align*}
		U(\ub):=\io
			\xi(t)\mathbb C(\chi^{k-1,t})\e^t(\ub):\e^t(\ub)
			+\xi(t)\frac{\ub-2\ub^{k-2,t}+\ub^{k-3,t}}{\tau^2}\cdot\ub
			-\xi(t)\bell^{k-1,t}\cdot\ub\dx.
	\end{align*}
	over $\ub\in \db^{k-1}+\ac H^1(\Omega;\R^2)$,
	where
	\begin{align}
		\e^t(\ub):=\frac12\Big((\partial\ub)(\partial\Phi_t)^{-1}
			+\big((\partial\ub)(\partial\Phi_t)^{-1}\big)^\top\Big).
			\label{etDef}
	\end{align}
	By using the calculation (here $S(\mathbf A):=\frac12(\mathbf A+\mathbf A^\top)$ and
	$\mathbf B:=\partial\Phi_t$)
	\begin{align*}
		&\mathbf C\e^t(\ub):\e^t(\ub)\\
		&\qquad=\mathbf C\e(\ub):\e(\ub)
			-\mathbf C S((\partial\ub)(1-B)):S(\partial\ub)
			-\mathbf C S((\partial\ub)\mathbf B):S((\partial\ub)(1-\mathbf B))\\
		&\qquad\geq \mathbf C\e(\ub):\e(\ub)
			-c|\partial\ub|^2|1-\mathbf B|-c|\partial\ub|^2|\mathbf B||1-\mathbf B|
	\end{align*}
	and Korn's inequality, we find a $t^*>0$ and constants $c_0,c_1>0$ such that for all $t\in[0,t^*]$ and $\ub\in \db^{k-1}+\ac H^1(\Omega;\R^2)$
	$$
		\io\xi(t)\mathbb C(\chi^{k-1,t})\e^t(\ub):\e^t(\ub)\dx\geq c_0\|\partial\ub\|_{L_2}^2-c_1\|\ub\|_{L_2}^2.
	$$
	Then the higher integrability result from \cite{FKR13}
	shows that there exists a constant $p>2$ independent of $t$
	such that $\ub^{k-1,t}\in W_p^1(\Omega;\R^2)$ and
	$\|\ub^{k-1,t}\|_{W_p^1(\Omega;\R^2)}$ is uniformly bounded in $t\in[0,t^*]$.
	In combination with \eqref{discUConv} we see that
	\begin{align}
		\ub^{k-1,t}\to\ub^{k-1,0}\quad\text{ strongly in }W_q^1(\Omega;\R^2)\text{ as }t\searrow 0\text{ for all }q\in[2,p).
		\label{discUhighConv}
	\end{align}
	Furthermore we deduce from \eqref{discChiConv} by the Sobolev embeddings in 2D
	\begin{align}
		&\frac{\chi^{k-1,t}-\chi^{k-1,0}}{t}\to \dot\chi^{k-1}\quad
		\text{ strongly in }L_q(\Omega)
		\text{ for all }q\in[1,\infty).
		\label{discULqConv}
	\end{align}
	and from \eqref{discUConv}
	\begin{align}
		&\frac{\e^t(\ukkt)-\e(\ub^{k-1,0})}{t}\to \e(\dot{\uu}^{k-1})-\frac12\big((\partial \uu^{k-1,0})(\partial X)+((\partial \uu^{k-1,0})(\partial X))^\top\big)
			=:\dot\e_X(\dot\uu^{k-1})\notag\\
		&\qquad\qquad\qquad\qquad\qquad\text{ strongly in }L_2(\Omega;\R^{2\times 2}).
			\label{etConv}
	\end{align}
	The damage variational inequality \eqref{discrDamageVI}
	can now be rewritten in the abstract form \eqref{eq:VI_weak}
	by setting $W_\Omega$ in the energy \eqref{eq:energy} as follows
	\begin{align*}
		 W_{\Omega}(x,y):={}&-\frac{1}{\tau}\chi^{k-1}(x)y+\frac{1}{2}\big(\mathsf c_1(y)+\mathsf c_2'(\chi^{k-1}(x))y)\CC\e(\ub^{k-1}(x)):\e(\ub^{k-1}(x))\\
			&+g_1(y)+g_2'(\chi^{k-1}(x))y.
	\end{align*}
	Note that $W_{\Omega}(x,\cdot)$ is convex in our discretisation scheme
	and that
	\begin{align*}
		w_X^t(x,y)
		={}&-\frac{1}{\tau}\chi^{k-1,t}(x)+\frac{1}{2}\big(\mathsf c_1'(y)+\mathsf c_2'(\chi^{k-1,t}(x))\big)\CC\e^t(\ub^{k-1,t}(x)):\e^t(\ub^{k-1,t}(x))\\
			&+g_1'(y)+g_2'(\chi^{k-1,t}(x)).
	\end{align*}
	Recall that $w_X^0(x,y)=w(x,y)=\partial_y W_\Omega(x,y)$.
	
	With the help of the convergence properties
	\eqref{discUConv}-\eqref{discChiConv}, \eqref{discUhighConv}, \eqref{etConv}
	and \eqref{discULqConv}, we see that Assumptions (A1)-(A3)
	are fulfilled with
	\begin{align*}
		\partial_y w(x,y)={}&\frac{1}{2}\mathsf c_1''(y)\CC\e(\ub^{k-1}(x)):\e(\ub^{k-1}(x))+g_1''(y),\\
		\dot w_{X}(x,y)={}&-\frac{1}{\tau}\dot\chi^{k-1}(x)
			+\frac{1}{2}\mathsf c_2''(\chi^{k-1}(x))\dot\chi^{k-1}(x)\CC\e(\ub^{k-1}(x)):\e(\ub^{k-1}(x))\\
			&+\mathsf c_2'(\chi^{k-1}(x))\CC\e(\ub^{k-1}(x)):\dot\e_X(\dot\ub^{k-1}(x))
			+g_2''(\chi^{k-1}(x))\dot\chi^{k-1}(x).
	\end{align*}
	Applying Corollary \ref{cor:dotU} yields existence of the strong material derivative
	$\dot\chi^k$ which satisfies the following variational inequality:
  \begin{align}
  	\left.
  	\begin{aligned}
	 		&\dot\chi^k\in S^{k}\text{ and }
		 		\forall \ph\in S^{k}:\\
	 		&\int_\Omega \nabla\dot \chi^k\cdot\nabla(\ph - \dot\chi^k)
	 			+ \frac{1}{\tau}\dot \chi^k(\ph - \dot\chi^k)
		 		+ \partial_y w(x,\chi^k)\dot\chi^k(\ph - \dot\chi^k)\dx\\
	   	&\quad
	   		\geq -\int_\Omega A'(0)\nabla \chi^k\cdot\nabla(\ph - \dot \chi^k)
	   		+ \xi'(0)\Big(\frac{\chi^k}{\tau}+w(x,\chi^k)\Big)(\ph - \dot \chi^k)\dx\\
	   	&\qquad
	   		-\int_\Omega\dot w_{X}(x,\chi^k)(\ph-\dot \chi^k)\dx.
  	\end{aligned}
  	\right\}
  	\label{discrChiLinVar}
  \end{align}
	with
	$$
		S^{k}:=T_{\chi^k}(K^{k-1})\cap kern(\C A(\chi^k))+\dot\chi^{k-1}
	$$
	and $A'(0)$ and $\xi'(0)$ are given in Lemma \ref{lem:convergence_comp}
	and $\C A=\C A_0$ is defined in \eqref{Adef}.
	\paragraph*{Material derivative for the $\ub^{\bm k}$-variable}
	\hspace*{0.1em}\\
	We only sketch the proof of the strong material derivative $\dot\ub^k$ in the following
	and make use of standard calculations.
	The main ingredient will be the	uniform boundedness of $\|\ub^{k,t}\|_{W_p^1(\Omega;\R^2)}$
	with respect to $t$ and for some fixed $p>2$.
	
	The perturbed and transported equation to \eqref{discrWeakU} is
	given by
	\begin{align}
		&\int_\Omega\xi(t)\frac{\ub^{k,t}-2\ub^{k-1,t}+\ub^{k-2,t}}{\tau^2}\ph+
			\xi(t)\mathbb C(\chi^{k,t})\e^t(\ub^{k,t}):\e^t(\ph)\dx
			=\int_\Omega\xi(t)\bell^{k,t}\cdot\ph\dx
			\label{pertDiscUEq}
	\end{align}
	for all $\ph\in \ac H^1(\Omega;\R^2)$,
	where $\xi(t)$ is defined in \eqref{abbrev}
	and $\e^t$ is defined in \eqref{etDef}.
	Therefore by testing \eqref{pertDiscUEq} and
	testing \eqref{discrWeakU} with $\ph=\ub^{k,t}-\ub^{k,0}-(\db^{k,t}-\db^{k,0})$ and subtracting the
	result, we obtain the sensitvity estimate
	\begin{align*}
		\|\ub^{k,t}-\ub^{k,0}\|_{H^1(\Omega;\R^2)}\leq ct.
	\end{align*}
	Thus we may choose a weak cluster point $\dot\ub^k\in H^1(\Omega;\R^2)$ such
	that for a subsequence
	\begin{align*}
		\frac{\ub^{k,t}-\ub^{k,0}}{t}\weaklim\dot\ub^k\quad\text{ weakly in }H^1(\Omega;\R^2).
	\end{align*}
	Considering difference quotient of \eqref{pertDiscUEq}
	and passing to the limit shows that $\dot\ub^k$
	is the weak solution of the following pde:
	\begin{align}
		&\int_\Omega\frac{\dot\ub^k-2\dot\ub^{k-1}+\dot\ub^{k-2}}{\tau^2}\ph
		+\Big(\mathbb C'(\chi^k)\dot\chi^k\e(\uu^k)
		+\mathbb C(\chi^k)\dot\e_X(\dot\uu^k)\Big):\e(\ph)
		+\mathbb C(\chi^k)\e(\uu^k):\dot\e_X(\ph)\dx\notag\\
		&\quad=\red-\int_\Omega\xi'(0)\frac{\ub^k-2\ub^{k-1}+\ub^{k-2}}{\tau^2}\ph
		+\xi'(0)\mathbb C(\chi^k)\e(\uu^k):\e(\ph)\dx+\black\int_\Omega\xi'(0)\fb^k\cdot\ph+\dot\fb^k\cdot\ph\dx
		\label{discrUlin}
	\end{align}
	for all $\ph\in \ac H^1(\Omega;\R^2)$ and $\dot\ub^k=\dot\db^k$ on $\partial\Omega$ where $\dot\db^k=\partial_X\db^k$
	and $\dot\fb^k=\partial_X\fb^k$.
	Here, $\dot\e_X$ is defined in \eqref{etConv}.
	
	Furthermore, it is not hard to see that the solution $\dot\ub^k$ is unique for
	given functions $\ub^k$, $\ub^{k-1}$, $\ub^{k-2}$, $\dot\ub^{k-1}$,
	$\dot\ub^{k-2}$, $\chi^{k}$, $\dot\chi^{k}$,
	$\fb^k$, $\dot\fb^k$ and $\dot\db^k$.
	Indeed, given to weak solutions $\dot\ub_1^k$ and $\dot\ub_2^k$
	of \eqref{discrUlin} we find after
	subtraction
	\begin{align*}
		&\int_\Omega\frac{1}{\tau^2}(\dot\ub_1^k-\dot\ub_2^k)
			+\mathbb C(\chi^k)\e(\dot\uu_1^k-\dot\uu_2^k):\e(\ph)\dx
			=0
	\end{align*}
	Testing with $\ph=\dot\ub_1^k-\dot\ub_2^k$ yields uniqueness.
	
	Finally, subtracting from the difference quotient taken from
	\eqref{pertDiscUEq} the equation \eqref{discrUlin} and testing with
	$\ph=\frac{\ub^{k,t}-\ub^{k,0}}{t}-\dot\ub^k-\Big(\frac{\db^{k,t}-\db^{k,0}}{t}-\dot\db^k\Big)$
	(the $\db$-terms are necessary to achieve $0$-boundary conditions for the test-function),
	we find via a limit passage
	$$
		\frac{\ub^{k,t}-\ub^{k,0}}{t}-\dot\ub^k \to 0\quad\text{ strongly in }H^1(\Omega;\R^2)\text{ as }t\searrow 0.
	$$
	\paragraph*{Optimality system}
	\hspace*{0.1em}\\
	We conclude with a necessary optimality system.
	Let a Lipschitz domain $\Omega\subseteq D$ with its state $\zb=\{\ub^k,\chi^k\}_{k=0}^N$
	be a minimiser of $J$ from \eqref{frakJDef}.
	Given an arbitrary vector field $X\in C^1_c(D,\R^d)$ we obtain the associated flow $\Phi_t$,
	the perturbed domain $\Omega_t:=\Phi_t(\Omega)$,
	the transported perturbed solution $\zb^t=\{\ub^{k,t},\chi^{k,t}\}_{k=0}^N$
	and
	\begin{align*}
		J(\Omega_t)=\frac{\lambda_\ub}{2}\sum_{k=1}^N\int_\Omega\xi(t)|\ub^{k,t}-\ub_r^k\circ\Phi_t|^2\dx
			+\frac{\lambda_\chi}{2}\sum_{k=1}^N\int_\Omega\xi(t)|\chi^{k,t}-\chi_r^k\circ\Phi_t|^2\dx.
	\end{align*}
	Due to the existence of the material derivatives $\dot\ub^k$ and $\dot\chi^k$ for $k=1,\ldots,N$,
	we know that the Eulerian semi-derivatives
	of $J$ at $\Omega$ exist and that $d J(\Omega)(\cdot)$ is positively 1-homogeneous by Lemma \ref{lemma:pos1hom}.
	Therefore a necessary optimality condition for shapes which minimises $J$ is given by the condition that $d J(\Omega)(X)\geq 0$ for all
	$X\in C^1_c(D, \R^2)$.
	By calculating the Eulerian semi-derivative of $J$ and using the relations for
	the material derivatives above, we have proven the following results:
	\begin{proposition}
	\label{prop:optimality}
		Under the assumption (D1) the optimality condition $d J(\Omega)(X)\geq 0$
		for all	$X\in C^1_c(D, \R^2)$
		is given in the volume expression of the shape derivative by
		\begin{align*}
			0\leq {}&\frac{\lambda_\ub}{2}\sum_{k=1}^N\int_\Omega\xi'(0)|\ub^{k}-\ub_r^k|^2\dx
				+\frac{\lambda_\chi}{2}\sum_{k=1}^N\int_\Omega\xi'(0)|\chi^{k}-\chi_r^k|^2\dx\\
				&+\lambda_\ub\sum_{k=1}^N\int_\Omega(\ub^{k}-\ub_r^k)\cdot(\dot\ub^k-\partial_X\ub_r^k)\dx
				+\lambda_\chi\sum_{k=1}^N\int_\Omega(\chi^{k}-\chi_r^k)(\dot\chi^k-\partial_X\chi_r^k)\dx,
		\end{align*}
		where for all $k=1,\ldots,N$:
		\begin{align*}
			&\text{$\ub^k$ fulfills  \eqref{discrWeakU} with $\ub^k=\db^k$ on $\partial\Omega$},
			&&\text{$\chi^k$ fulfills  \eqref{discrDamageVI}},\\
			&\text{$\dot\ub^k$ fulfills \eqref{discrUlin} with $\dot\ub^k=\dot\db^k$ on $\partial\Omega$},
			&&\text{$\dot\chi^k$ fulfills \eqref{discrChiLinVar}.}
		\end{align*}
	\end{proposition}

\appendix

\section{Polyhedricity of upper obstacle sets in $H^{1}(\Omega)$}
In the remaining part of this subsection we will sketch the proofs for the characterisation
of the tangential and normal cones as well as of the polyhedricity of $K_\psi$
for reader's convenience
since such obstacles sets are usually considered in the space
\red $$
	\ac H^1(\Omega) := \overline{C^\infty_c(\Omega)}^{\|\cdot\|_{H^1}}
$$
\black
in the literature.
The adaption to $H^1(\Omega)$ requires some careful modifications in the proofs.

Furthermore we denote with $M_+(\ol\Omega)$ the Radon measures on $\ol\Omega$.
The Riesz representation theorem for local compact Hausdorff spaces
(see \cite[Theorem VIII.2.5]{El00}) states that
for each non-negative functional $I:C(\ol\Omega)\to\R$
there exists a unique Radon measure $\mu\in M_+(\ol\Omega)$
such that for all $f\in C(\ol\Omega)$
\begin{align}
	I(f)=\int_{\ol\Omega}f\mathrm d\mu.
	\label{Riesz}
\end{align}
In the sequel we will use the following notation for the half space
\begin{align*}
	H_+^1(\Omega):=\big\{v\in H^1(\Omega):\,v\geq 0\aein\Omega\big\}.
\end{align*}
With the help of the Riesz representation theorem
we are now in the position to give a characterisation of
(cf. \cite[Chapter 6.4.3]{bosha}
for $\ac H^1(\Omega)$ instead of $H^1(\Omega)$)
\begin{align*}
	H^1(\Omega)_+^*:=\big\{I\in H^1(\Omega)^*:\,\langle I,v\rangle_{H^1(\Omega)}\geq 0\text{ for all }v\in H_+^1(\Omega)\big\},
\end{align*}
\red where $H^1(\Omega)^*$ denotes the topological dual space of $H^1(\Omega)$.\black
\begin{lemma}
\label{lemma:identification}
	We have
	\begin{align}
		H^1(\Omega)_+^*
		=\Big\{I\in H^1(\Omega)^*:\,\exists!\,\mu_I\,\in M_+(\ol\Omega),\,
			\forall v\in H^1(\Omega)\cap C(\ol\Omega),\,\langle I,v\rangle_{H^1(\Omega)}=\int_{\ol\Omega}v\,\mathrm d\mu_I\Big\}.
		\label{H1posChar}
	\end{align}
\end{lemma}
\begin{proof}
	Let $I:H^1(\Omega)\to\R$ be a non-negative\red, linear and continuous \black functional.
	Then \red in particular \black the restriction $I|_{H^1(\Omega)\cap C(\ol\Omega)}$ is a
	non-negative \red and linear \black functional on the space $H^1(\Omega)\cap C(\ol\Omega)=:Y$.
	
	Now let $y\in Y$ be arbitrary. Then
	$y^+:=\max\{0,y\}$ and $y^-:=\min\{0,y\}$ (defined in a pointwise sense) are also in $Y$ and we find
	by non-negativity of $L\red :=I|_Y\black$:
	\begin{align*}
		|Ly|=|L(y^++y^-)|=|\underbrace{L(y^+)}_{\geq 0}+\underbrace{L(y^-)}_{\leq 0}|
		\leq{}& |\underbrace{L(y^+)}_{\geq 0}\underbrace{-L(y^-)}_{\geq 0}|\\
		\leq{}& |L(y^+-y^-)|=L(|y|)\\
		={}&\underbrace{L(|y|-\mathds{1}\|y\|_\infty)}_{\leq 0}+\|y\|_\infty L(\mathds 1)\\
		\leq{}&\|y\|_\infty L(\mathds 1),
	\end{align*}
	\red where $\mathds 1$ denotes the constant mapping with $\mathds 1(x):=1$.
	\black
	Thus $I|_Y$ is continuous in the $C(\ol\Omega)$-topology.
	Since $Y$ is also dense in $C(\ol\Omega)$ the functional
	$I|_Y$ has a unique continuous and non-negative
	extension $\tilde I:C(\ol\Omega)\to\R$ over $C(\ol\Omega)$.
	By the Riesz representation theorem (see \eqref{Riesz}) we find a \red unique \black
	$\mu\in M_+(\ol\Omega)$ such that $\red \tilde I\black(v)=\int_{\ol\Omega}v\,\mathrm d\mu$
	for all $v\in C(\ol\Omega)$.
	
	Conversely, let $I$ be in the set on the right-hand side of \eqref{H1posChar}.
	Then we know $\langle I,v\rangle_{H^1(\Omega)}=\int_{\ol\Omega}v\,\mathrm d\mu_I\geq 0$
	for all $v\in Y_+:=\{v\in Y:\,v\geq 0$ pointwise in $\ol\Omega\}$.
	So by density of $Y_+$ in $H_+^1(\Omega)$ we obtain $I\in H^1(\Omega)_+^*$.
\end{proof}
\begin{remark}
	Note that, by an abuse of notation,
	the right-hand side of \eqref{H1posChar}
	is sometimes written as
	$H^1(\Omega)^*\cap M_+(\ol\Omega)$
	(see, e.g., \cite[Chapter 6]{bosha}).
\end{remark}
For the notion of \textit{capacity of a set}, \textit{quasi-everywhere (q.e.)}
and \textit{quasi-continuous representant}
\red we refer \black to \cite[Chapter 3.3]{HP05}.
The following result is an extension of \eqref{Riesz}
valid for elements from $H^1(\Omega)_+^*$.
\begin{lemma}
\label{lemma:muInt}
	For all $ I\in H^1(\Omega)_+^*$ and all $f\in H^1(\Omega)$
	\red there exists \black
	$\tilde f\in L_1(\ol\Omega,\mu_I)$ and \red we have\black
	\begin{align}
		\langle I,f\rangle_{H^1(\Omega)}=\int_{\ol\Omega} \tilde f\mathrm d\mu_I,
		\label{muFInt}
	\end{align}
	where $\tilde f$ (defined on $\ol\Omega$) denotes a quasi-continuous representative of $f$
	and $\mu_I$ the measure from \eqref{H1posChar} of Lemma \ref{lemma:identification}.
\end{lemma}
\begin{proof}
	The proof of this lemma requires some modifications
	of \cite[Lemma 6.56]{bosha} and references therein
	which were designed \red for \black the situation $V=\ac H^1(\Omega)$.
	In our case we will need the following auxiliary results:
	\begin{itemize}
		\item[(a)]
			For an arbitrary $D\subseteq\R^d$ the capacity of $D$ calculates as
			\begin{align*}
				\text{cap}(D)
				=\inf\big\{\|v\|_{H^1(\R^d)}^2:\,v\in H^1(\R^d)\text{ and }v\geq 1\aein\text{a neighborhood of }D\big\}.
			\end{align*}
			See \cite[Proposition 3.3.5]{HP05} for a proof.
		\item[(b)]
			Any function $f\in H^1(\Omega)$ can be approximated by a sequence
			$\{f_n\}\subseteq C_c^\infty(\R^d)$ in the sense that $f_n\to f$ in $H^1(\R^d)$
			as $n\to\infty$ by extending $f$ to $\R^d$ with compact support
			and then uses an approximation argument via Friedrichs mollifiers.
	\end{itemize}
	The proof carried out in the following steps on the basis of \cite[Lemma 6.56]{bosha}
	and the references therein (see also \cite[Th\'eor\`eme 3.3.29]{HP05} for the case $V=H^1(\R^d)$):
	\vspace*{0.5em}\\\textit{\underline{Claim 1:}
		There exists a sequence $\{f_n\}\subseteq C_c^\infty(\R^d)$
		s.t. $f_n|_{\Omega}\to \tilde f$ in $H^1(\Omega)$ and q.e. in $\ol\Omega$}\vspace*{0.5em}\\
	Let $\{f_n\}$ be given by (b).
	By resorting to a subsequence (we omit the subscript) we may find
	$\|f_n-f\|_{H^1(\R^d)}\leq 2^{-n}n^{-1}$ and therefore
	\begin{align}
		\sum_{n=1}^\infty 4^{n+1}\|f_{n+1}-f_{n}\|_{H^1(\R^d)}^2
			\leq\sum_{n=1}^\infty 4^{n+1}(\|f_{n+1}-f\|_{H^1(\R^d)}+\|f_{n}-f\|_{H^1(\R^d)}\big)^2<+\infty.
		\label{fSeq}
	\end{align}
	We define
	\begin{align*}
			B_n:=\big\{x\in\R^d:\,|f_{n+1}(x)-f_n(x)|\geq 2^{-n}\big\}.
	\end{align*}
	Since $|f_{n+1}-f_n|$ is a continuous with compact support in $\R^d$,
	the set $B_n$ is compact and
	\begin{align*}
		2^{n+1}|f_{n+1}-f_n|\geq 1\;\text{ holds in a neighborhood of }B_n.
	\end{align*}
	Thus by (a)
	\begin{align*}
		\capa(B_n)\leq 4^{n+1}\|f_{n+1}-f_n\|_{H^1(\R^d)}^2.
	\end{align*}
	Using this estimate, the sub-additivity of the capacity (see \cite[Remarque 3.3.10]{HP05})
	and \eqref{fSeq},
	we obtain:
	\begin{align}
		\capa\big(\bigcup_{k=n}^\infty B_k\big)
		\leq\sum_{k=n}^\infty\capa(B_k)
		\leq\sum_{k=n}^\infty 4^{n+1}\|f_{n+1}-f_n\|_{H^1(\R^d)}^2\to 0\quad\text{as }n\to\infty.
		\label{capEst}
	\end{align}
	Now let $n\in\N$ and $x\in\ol\Omega\setminus \bigcup_{k=n}^\infty B_k$
	be arbitrary.
	Then
	$\{f_k(x)\}_{k\geq n}$ is a Cauchy sequence since for all $m\geq n$:
	$$
		|f_m(x)-f_n(x)|\leq\sum_{k=n}^{m-1}|f_{k+1}(x)-f_k(x)|
		\leq \sum_{k=n}^{m-1}2^{-k}.
	$$
	We denote the limit with $\tilde f(x)$ and gain for all $N,K\geq n$:
	$$
		|\tilde f(x)-f_N(x)|
			\leq\underbrace{|\tilde f(x)-f_{K+1}(x)|}_{\to 0\text{ as }K\to\infty}
			+\sum_{k=N}^K \underbrace{|f_{k+1}(x)-f_k(x)|}_{\leq 2^{-k}\text{ since }x\in\ol\Omega\setminus \bigcup_{k=n}^\infty B_k}
	$$
	\red Passing to the limit \black $K\to\infty$ then shows
	$$
		|\tilde f(x)-f_N(x)|
			\leq \sum_{k=N}^\infty 2^{-k}.
	$$
	This estimate implies that $\{f_N\}_{N\geq n}$
	converges uniformly to $\tilde f$ on the set $\ol\Omega\setminus \bigcup_{k=n}^\infty B_k$.
	Due to \eqref{capEst} we obtain Claim 1.
	\vspace*{0.5em}\\\textit{\underline{Claim 2:}
		If $\capa(A)=0$ for a Borel set $A\subseteq\ol\Omega$ than $\mu_I(A)=0$.
		}\vspace*{0.5em}\\
	Let $\varepsilon>0$ be arbitrary.
	By (a) we find a function $u\in H^1(\R^d)$ such that
	$\|u\|_{H^1(\Omega)}<\varepsilon$ and
	$u\geq 1$ a.e. on $A_\varepsilon$ where $A_\varepsilon$
	is a neighborhood of $A$.
	Thus there exists a Lipschitz function
	$f_\varepsilon:\R^d\to[0,1]$ such that
	\begin{align*}
		f_\varepsilon(x)=
		\begin{cases}
			0&\text{ if }x\in\R^d\setminus A_\varepsilon,\\
			\in(0,1)&\text{ if }x\in A_\varepsilon\setminus A,\\
			1&\text{ if }x\in A.
		\end{cases}
	\end{align*}
	Then $f_\e-u\leq 0$ a.e. in $\Omega$ and by Lemma \ref{lemma:identification}
	\begin{align*}
		\mu_I(A)=\int_{A}\mathds 1\mathrm d\mu_I
		\leq\int_{\ol\Omega}f_\e\,\mathrm d\mu_I
		=\langle I,f_\e\rangle_{H^1(\Omega)}
		={}&\langle I,u\rangle_{H^1(\Omega)}
			+\underbrace{\langle I,f_\e-u\rangle_{H^1(\Omega)}}_{\leq\,0\text{ since }f_e\leq\,u\aein\Omega}\\
		\leq{}&\langle I,u\rangle_{H^1(\Omega)}\\
		\leq{}&\varepsilon\|I\|_{H^1(\Omega)^*}.
	\end{align*}
	\red Passing to the limit \black $\varepsilon\searrow 0$ yields to claim.
	\vspace*{0.5em}\\\textit{\underline{Claim 3:}}
		$f_n\to \tilde f$ in $L^1(\ol\Omega,\mu_I)$\\
	Lemma \ref{lemma:identification} implies for every $n,m\in\N$
	\begin{align}
		\int_{\ol\Omega}|f_n-f_m|\mathrm d\mu_I
		=\langle I,|f_n-f_m|\rangle_{H^1(\Omega)}
		\leq\|I\|_{H^1(\Omega)^*}\|f_n-f_m\|_{H^1(\Omega)},
		\label{fDiffEst}
	\end{align}
	where $f_n$ is the approximation sequence from Claim 1.
	Since $f_n\to f$ in $H^1(\Omega)$ we obtain from \eqref{fDiffEst}
	that $\{f_n\}$ is a Cauchy sequence in $L^1(\ol\Omega,\mu_I)$.
	Thus there exists a limit element $\tilde g\in L^1(\ol\Omega,\mu_I)$
	and a subsequence (we omit the subscript) such that
	$f_n\to\tilde g$ in $L^1(\ol\Omega,\mu_I)$ and pointwise $\mu_I$-a.e.
	on $\ol\Omega$.
	However, by Claim 1, we already know that $f_n$ converges q.e. to $\tilde f$
	on $\ol\Omega$ and, by Claim 2, we find that this covergence is
	also $\mu_I$-a.e.
	Thus $\tilde f=\tilde g$ $\mu_I$-a.e.
	\vspace*{0.5em}\\\textit{\underline{Conclusion:}}\\
	Finally, Lemma \ref{lemma:identification} shows for every $n\in\N$
	\begin{align*}
		\langle I,f_n\rangle_{H^1(\Omega)}=\int_{\ol\Omega} f_n\mathrm d\mu_I.
	\end{align*}
	With the properties proven above we can pass to the limit $n\to\infty$
	and obtain \eqref{muFInt}.
\end{proof}
We are now in a position to characterise the
tangential and normal cones in $K_\psi$.
\begin{proof}[Proof of Theorem \ref{thm:TN}]
	From the definitions \eqref{clarke_cone}-\eqref{normal_cone} we see that
	$$
		T_y(K_\psi)=T_{y-\psi}(K),\quad N_y(K_\psi)=N_{y-\psi}(K)
	$$
	with $K:=\{w\in H^1(\Omega):\;w\leq 0\text{ a.e. in }\Omega\}$.
	Thus it suffices to prove the assertion for $K_\psi=K$.
	
	We firstly prove \eqref{normalSpace}.
	
	\noindent
	``$\subseteq$'':
	Let $I\in N_y(K)$. Then by using definition  \eqref{normal_cone} and choosing $v=y+w$ for an arbitrary $w\in H^1(\Omega)$ with $w\leq 0$ a.e. we obtain
	$\langle I,w\rangle_{H^1(\Omega)}\leq 0$.
	Thus $I\in H^1(\Omega)_+^*$ and by Lemma \ref{lemma:identification}
	we find the associated measure $\mu_I$ from \eqref{H1posChar}.
	On the other hand by choosing $v=\psi$ and $v=2y$ in \eqref{normal_cone} yields
	$\langle I,y\rangle_{H^1(\Omega)}=0$.
	From Lemma \ref{lemma:muInt} we obtain
	\begin{align}
		\int_{\ol\Omega} \tilde y\,\mathrm d\mu_I=0\quad
		\text{with a quasi-continuous representant $\tilde y$ of }y.
		\label{intZero}
	\end{align}
	Since $y\leq 0$ a.e. in $\Omega$ we find $\tilde y\leq 0$ q.e. in $\ol\Omega$
	(see \cite[Remarque 3.3.6]{HP05}).
	This implies in combination with \eqref{intZero} that
	$\int_{\ol\Omega}|\tilde y|\,\mathrm d\mu_I=0$.
	Thus $\int_{\{\tilde y < 0\}}|\tilde y|\,\mathrm d\mu_I=0$
	and therefore $\mu_I(\{\tilde y < 0\})=0$.

	\noindent
	``$\supseteq$'':
	Let $I\in H^1(\Omega)_+^*$ with $\mu_I(\{\tilde y<0\})=0$.
	Now let $v\in K$ be arbitrary. The splitting $v=\max\{v,y\}+\min\{0, v-y\}$
	implies
	\begin{align*}
		\langle I,v-y\rangle_{H^1(\Omega)}={}&
			\langle I,\max\{v,y\}-y\rangle_{H^1(\Omega)}
			+\underbrace{\langle I,\min\{0, v-y\}\rangle_{H^1(\Omega)}}_{\leq 0}\\
		\leq{}& \int_{\{\tilde y=0\}}\max\{\tilde v,\tilde y\}-\tilde y\,\mathrm d\mu_I
			+\underbrace{\int_{\{\tilde y<0\}}\max\{\tilde v,\tilde y\}-\tilde y\,\mathrm d\mu_I}_{=0\text{ since }\mu_I(\{\tilde y<0\})=0}\\
		\leq{}& \int_{\{\tilde y=0\}}\underbrace{\max\{\tilde v,0\}}_{=0\text{ since }v\in K}\,\mathrm d\mu_I=0.
	\end{align*}
	Hence $I\in N_y(K)$.
	
	Now we prove \eqref{tangSpace}.
	By applying the bipolar theorem as in \eqref{bipolar} as well as Lemma \ref{lemma:muInt},
	we find
	\begin{align*}
		&T_y(K)=\Big\{u\in H^1(\Omega):\,\int_{\ol\Omega} \tilde u\,\mathrm d\mu_I\leq 0\text{ for all }I\in H^1(\Omega)_+^*\text{ with }
			\mu_I(\{\tilde y<0\})=0\Big\}\\
		&\quad=\Big\{u\in H^1(\Omega):\,\int_{\{\tilde y=0\}}\tilde u\,\mathrm d\mu_I\leq 0\text{ for all }I\in H^1(\Omega)_+^*\text{ with }
			\mu_I(\{\tilde y<0\})=0\Big\}.
	\end{align*}
	From this representation we see that the ``$\supseteq$''-inclusion in \eqref{tangSpace}
	is fulfilled.
	Conversely, let $u\in T_y(K)$. By definition of $T_y(K)$ given in \eqref{tangent_cone} we find
	a sequence $v_n\in K$ and $t_n>0$ such that
	$t_n(v_n-y)\to u$ in $H^1(\Omega)$ as $n\to\infty$. This implies for a subsequence
	(we omit the subindex)
	$t_n(\tilde v_n-\tilde y)\to \tilde u$ q.e. in $\ol\Omega$.
	Since $v_n\in K$ we see that
	\begin{align*}
		t_n(\tilde v_n-\tilde y)=t_n \tilde v_n\leq 0\text{ q.e. on }\{\tilde y=0\}.
	\end{align*}
	Thus $\tilde u\leq 0$ q.e. on $\{\tilde y=0\}$.
\end{proof}
\begin{proof}[Proof of Theorem \ref{theorem:Kpolyhedr}]
	Let $y$ and $w$ as in \eqref{polyhedric_set} and let
	$v\in T_y(K_\psi)\cap [w]^\perp$.
	Then there exists a sequence $v_n\to v$ strongly in $H^1(\Omega)$
	such that $v_n\in C_y(K_\psi)$.
	Define
	$$
		\red \hat v_n\black:=\max\{v_n,v\}.
	$$
	By resorting to quasi-continuous representants we find by Theorem \ref{thm:TN}
	\begin{align*}
		v\leq 0\text{ q.e. in }\{y=\red\psi\black\}\quad\text{ and }\quad v_n\leq 0\text{ q.e. in }\{y=\red\psi\black\}
	\end{align*}
	and thus
	\begin{align*}
		\red \hat v_n\black\leq 0\text{ q.e. in }\{y=\red\psi\black\}.
	\end{align*}
	Moreover by definition of $\red\hat v_n\black$
	\begin{align*}
		v-\red\hat v_n\black\leq 0\text{ q.e. in }\Omega.
	\end{align*}
	Invoking Theorem \ref{thm:TN} again yield $\red\hat v_n\black\in T_y(K_\psi)$
	and $v-\red\hat v_n\black\in T_y(K_\psi)$.
	Since $w\in N_y(K_\psi)$ we see by \eqref{normal_cone_repr} that
	\begin{align*}
		\langle w, \red\hat v_n\black\rangle\leq 0
		\quad\text{and}\quad
		\langle w, v-\red\hat v_n\black\rangle\leq 0.
	\end{align*}
	Taking also $\langle w, v\rangle=0$ into account we obtain from above that
	$\langle w, \red\hat v_n\black\rangle=0$. Thus $\red\hat v_n\black\in C_y(K_\psi)\cap[w]^\perp$.
	\red Since $\hat v_n$ converges strongly to $v$ as $n\to\infty$, we have proven
	$$
		T_y(K_\psi)\cap [w]^\perp\subseteq \ol{C_y(K_\psi)\cap[w]^\perp}.
	$$
	Noticing that the ``$\supseteq$''-inclusion is always satisfied finishes the proof.
\end{proof}

\section{Eulerian semi and shape derivatives}
\label{sec:Euler}
We recall some preliminaries from shape optimisation theory.
For more details we refer to \cite{DelZol11}.

Let $X:\R^d\rightarrow \R^d$ be a vector field 
satisfying a global Lipschitz condition: there is a constant $L>0$ 
such that 
$$ |X(x)-X(y)| \le L|x-y| \quad \text{ for all } x,y\in \R^d .$$ 
Then  we 
associate with $X$ the flow $\Phi_t$ by solving for all $x\in \R^d$
\begin{align}
	\frac{d}{dt} \Phi_t(x) = X(\Phi_t(x)) \;\text{ on } [-\tau,\tau], \quad \Phi_0(x)=x.
	\label{PhiFlow}
\end{align}
The global existence of the flow $\Phi:\R\times \R^d \rightarrow \R^d$
is ensured by the theorem of Picard-Lindel\"of.

Subsequently, we restrict ourselves to a special class  of vector fields, namely 
$C^k$-vector fields with compact support in some fixed set. 
To be more precise for a fixed  open set $D\subseteq \R^d$, we consider vector 
fields belonging to  $C^k_c(D,\R^d)$.
We equip the space $C^k_c(D,\R^d)$ respectively $C^\infty_c(D,\R^d)$ with the topology induced 
by the following  family of semi-norms: for each compact $K\subseteq D$ and muli-index 
$\alpha\in \N^d$ with $|\alpha| \le k$ we define
$ \|f\|_{K,\alpha} := \sup_{x\in K} |\partial^\alpha f(x) |. $
With this familiy of semi-norms the space $C^k_c(D,\R^d)$ becomes a locally convex 
vector space.  

Next, we recall the definition of the Eulerian semi-derivative.

\begin{definition}\label{def1} 
	Let $D\subseteq \R^d$ be an open set.  Let $J:\Xi \rightarrow \R$ be a shape function defined on a set $\Xi$ of subsets
	of $D$ and fix  $k\ge 1$. Let $\Omega\in \Xi$ and $X \in C^k_c(D,\R^d)$ be such that 
	$\Phi_t(\Omega) \in \Xi$ for all 
	$t >0$ sufficiently small.   
	Then the  Eulerian semi-derivative of $J$ at $\Omega$ in direction $X$ is defined by
	\begin{equation}
	dJ(\Omega)(X):= \lim_{t \searrow 0}\frac{J(\Phi_t(\Omega))-J(\Omega)}{t}.
	\end{equation}
	\begin{itemize}
		\item[(i)] The function $J$ is said to be \textit{shape differentiable} at $\Omega$ if  $dJ(\Omega)(X)$ exists for all $X \in C^\infty_c(D,\R^d)$ and 
		$ X     \mapsto dJ(\Omega)(X) $ 
		is linear and continuous on $C^\infty_c(D,\R^d)$.
		\item[(ii)]  The smallest integer  $k\ge 0$ for which $X \mapsto dJ(\Omega)(X)$ is continuous with respect to the $C^k_c(D,\R^d)$-topology is called the order of $dJ(\Omega)$. 
	\end{itemize}
\end{definition}
The set $D$ in the previous definition is usually called hold-all domain or 
hold-all set or universe.

In the case that the state system is given as a solution of a variational
inequality we cannot expect $dJ(\Omega)(X)$ to be linear in $X$.
However we have the following general result:
\begin{lemma}
\label{lemma:pos1hom}
	Suppose that the Eulerian semi-derivative $dJ(\Omega)(X)$ exists
	for all $X \in C^k_c(D,\R^d)$.
	Then $dJ(\Omega)(\cdot)$ is positively 1-homogeneous.
\end{lemma}
\begin{proof}
	Let $\lambda>0$ be arbitrary. We write $\Phi_t^{\lambda X}$ for the flow
	induced by $\lambda X$.
	By definition \eqref{PhiFlow}, we see
	that $\Phi_t^{\lambda X}$ and $\Phi_{\lambda t}^{X}$
	solve
	$$
		\frac{d}{dt} \Phi_t^{\lambda X}(x) = \lambda X(\Phi_t^{\lambda X}(x)),\quad
		\frac{d}{dt} \Phi_{\lambda t}^{X}(x) = \lambda X(\Phi_{\lambda t}^{X}(x))
	$$
	as well as $\Phi_0^{\lambda X}(x)=x$ and $\Phi_0^{X}(x)=x$.
	Uniqueness of the flow implies $\Phi_t^{\lambda X}=\Phi_{\lambda t}^{X}$.
	Finally,
	$$
		dJ(\Omega)(\lambda X)
		=\lim_{t \searrow 0}\frac{J(\Phi_t^{\lambda X}(\Omega))-J(\Omega)}{t}
		=\lim_{t \searrow 0}\frac{J(\Phi_{\lambda t}^{X}(\Omega))-J(\Omega)}{t}
		=\lambda\,dJ(\Omega)(X).
	$$
\end{proof}
\red Ultimately the goal would be to find descent directions of a given shape function $J(\cdot)$, that is, finding
solutions of $\min_{X\in \mathcal H} dJ(\Omega)(X)$ in some Hilbert space $\mathcal H\subset C(\R^d,\R^d)$; cf. \cite{eigelsturm16}. 
Now Lemma \ref{lemma:pos1hom} tells us that it is sufficient to minimise over the unit sphere: 
     $$ \min_{X\in \mathcal H} dJ(\Omega)(X) =   \min_{\substack{X\in \mathcal H\\ \|X\|_{\mathcal H}=1}} dJ(\Omega)(X)$$
     which leads to a simplification of the minimisation problem; cf. \cite{sturm15}.  In the context of variational inequalities it rarely happens that the Eulerian semi-derivative
     is linear, however, the 1-homogeneity is valid as soon as the Eulerian semi-derivative exists.  
\black

The following result can be found for instance in \cite{DelZol11}:
\begin{lemma}\label{lemma:phit}
Let $D\subseteq \R^d$ be open and bounded and suppose $X\in C^1_c(D, \R^d)$.
	\begin{itemize} 
	\item[(i)] We have
	\begin{align*}
		\frac{ \partial \Phi_t - I}{t} \rightarrow & \partial X && \text{ strongly in } C(\overline D, \R^{d,d})\\
		\frac{ \partial \Phi_t^{-1} - I}{t} \rightarrow & - \partial X && \text{ strongly in } C(\overline D, \R^{d,d})\\
	         \frac{ \det(\partial \Phi_t) - 1}{t} \rightarrow & \divv(X) && \text{ strongly in } C(\overline D).
	\end{align*}
\item[(ii)] 	For all open sets $\Omega \subseteq D$ and all $\varphi\in W^1_\mu(\Omega)$, $\mu \ge 1$, we have 
          \begin{align}
            \frac{\varphi\circ \Phi_t  - \varphi}{t} \rightarrow & \nabla \varphi \cdot X
            && \text{ strongly in }  L_\mu(\Omega).
           \end{align}
           	\end{itemize}
\end{lemma}
	
\bibliographystyle{plain}
\bibliography{refs}

\end{document}